\documentclass[final,3p,times]{elsarticle}

\usepackage{amssymb}
\usepackage{amsmath}
\usepackage{amsthm}
\usepackage{hhline}
\usepackage{multirow}
\usepackage{xcolor}

\journal{}

\def\la {\langle}
\def\ra {\rangle}
\def\ui {{\bf u}^{n, i}}
\def\bu {{\bf u}}
\def\eiu {{\bf e}_{\bf u}^{i}}
\def\fui {{\bf e}_{\bf u}^{i-1}}

\def\epi {e_p^{i}}
\def\eip1 {{e_p^{i-1}}}
\def\bv {{\bf v }}

\begin{document}

\begin{frontmatter}



\title{A robust iterative scheme for the slightly compressible Darcy-Forchheimer equations}


\author[1]{Laura Portero}\ead{laura.portero@unavarra.es}
\author[1]{Andr\'es Arrar\'as}\ead{andres.arraras@unavarra.es}
\author[2]{Francisco J. Gaspar}\ead{fjgaspar@unizar.es}
\author[3]{Florin A. Radu}\ead{florin.radu@uib.no}

\affiliation[1]{organization={Institute for Advanced Materials and Mathematics (INAMAT2), Department of Statistics, Computer Science and Mathematics, Public University of Navarre},
	addressline={Campus Arrosad\'ia},
	city={Pamplona},	postcode={31006},
	country={Spain}
}

\affiliation[2]{organization={University Institute of Mathematics and Applications (IUMA), Department of Applied Mathematics, University of Zaragoza},
	addressline={Pedro Cerbuna 12},
	city={Zaragoza},
	postcode={50009},
		country={Spain}
}
	
\affiliation[3]{organization={Center for Modeling of Coupled Subsurface Dynamics, Department of Mathematics, University of Bergen},
	addressline={Allegaten 41},
	city={Bergen},
	postcode={5007},
	country={Norway}
}

\begin{abstract}
We study the slightly compressible Darcy-Forchheimer equations modeling gas flow in porous media, particularly in applications related to combustion processes. The equations are discretized in time using the backward Euler method and in space via a mixed finite element scheme. As a result, a nonlinear algebraic system is obtained at each time step.

We propose and analyze a general iterative linearization scheme for the efficient solution of such systems and study its convergence properties at the discrete level. The performance and robustness of the scheme are assessed through a series of numerical experiments. The method is compared with standard iterative solvers, and further tested on problems with discontinuous permeability fields. The results demonstrate its reliability and competitiveness in regimes characterized by strong nonlinear effects.

\end{abstract}



\end{frontmatter}

\newtheorem{theorem}{Theorem}
\newtheorem{lemma}{Lemma}
\newtheorem{remark}{Remark}



\section{Introduction}\label{sec:intro}

Darcy's law is the empirical linear model that describes how fluid flows through porous media when the flow is laminar and the inertial effects are negligible. In regimes characterized by higher velocities, for instance near wells and fractures or when considering highly permeable media, inertial effects are no longer negligible and Darcy's law may fail. To account for nonlinear inertial effects, Forchheimer extended Darcy's law by adding a quadratic nonlinear correction term. 

The complexity of the Darcy-Forchheimer model varies significantly depending on the fluid properties. In the incompressible stationary case, the fluid density is assumed to be constant and the variables do not depend on time, leading to a system of nonlinear elliptic equations \cite{gir:whe:2008,lop:mol:sal:2009}. Conversely, the slightly compressible evolutionary model is more relevant for gas flow modeling, where the density depends on the pressure via a state equation \cite{aul:blo:hoa:ibra:2009,kie:2015,rui:pan:2017}. There are many important applications behind, especially in connection with combustion in porous media: gas turbines, fuel cells, industrial drying, radial burners for heating, or waste
management, to mention a few (see also the review paper \cite{CombustionReview}). The mathematical well-posedness of these models has been studied during the last decades. Existence, uniqueness, and regularity results for the Darcy-Forchheimer equations were first derived in \cite{Fabrie_1989,Amirat_1991,kna:sum:2016}.

Due to the coupling and nonlinearities, the Darcy-Forchheimer model is very challenging to be solved numerically. In fact, there are numerous papers in the literature dealing with its numerical solution. 

The mixed finite element (MFE) method is a very popular technique for the discretization of such equations, since it is locally mass conservative and provides an accurate approximation of the flux variable. We mention \cite{kna:sum:2016,douglas1993,kim1999,ewing1999,summ_thesis,park2005} for MFE discretizations of the Darcy-Forchheimer model. Remarkably, the scheme proposed in \cite{gir:whe:2008} uses discontinuous constant elements for the velocity and discontinuous $\mathbb{P}_1$ Crouzeix--Raviart elements for the pressure. In addition, the nonlinear system stemming from the discretization is solved by means of an alternating-direction iterative algorithm inspired by the Peaceman--Rachford splitting scheme \cite{peaceman1955}. Extensions of this work to different MFE spaces are presented in \cite{lop:mol:sal:2009}. Another weak formulation, related to Raviart--Thomas and Brezi--Douglas--Marini elements, is proposed in \cite{pan:rui:2012}, where existence, uniqueness and regularity are proven for both the weak formulation and the corresponding discrete schemes. In \cite{kie:2015}, an expanded MFE discretization is proposed and analyzed.

On the other hand, block-centered finite difference schemes have also been applied for solving the Darcy-Forchheimer equations \cite{rui:pan:2012}. Such schemes can be obtained from the lowest order Raviart--Thomas MFE spaces with a suitable quadrature formulation, and extend the original method introduced for linear elliptic problems with diagonal and full tensor diffusion coefficients (see \cite{wei:whe:1988} and \cite{arb:whe:yot:1997}, respectively) to the nonlinear case. Subsequent generalizations to slightly compressible Darcy-Forchheimer flow are proposed in \cite{rui:pan:2017}.

Multipoint flux approximation methods represent an alternative choice for the Darcy-Forchheimer model. In \cite{xu:lia:rui:2017}, a multipoint flux MFE method is proposed for the  compressible model, considering the lowest order Brezzi--Douglas--Marini elements in combination with a suitable quadrature rule that allows for local velocity elimination.

In the context of finite volume discretizations, \cite{both2020DF} presents an efficient technique for Darcy-Forchheimer flow in deformable porous media (i.e., coupled with the quasi-static Biot equations). The generalized gradient flow structure of the model (see \cite{bothGradientFlow2019}) is used to analyze the scheme.

Discrete fracture-matrix models which couple Darcy flow in the porous matrix with Forchheimer flow within the fractures have also been proposed in the literature (cf. \cite{kna:rob:2014}). In \cite{Frih_2006,Frih_2008}, the lowest order Raviart--Thomas elements are combined with a domain decomposition technique in order to obtain numerical approximations. The nonlinear system resulting from the Forchheimer equation is solved by means of fixed-point iteration and quasi-Newton methods. Other efficient solvers based on geometric multigrid techniques have been proposed in \cite{arr:gas:por:rod:2019} for such models.

All of the aforementioned models and discretizations share a common feature: solving the nonlinear algebraic systems arising from the spatial (and sometimes also temporal) discretization remains a significant challenge. Newton method is frequently used due to its quadratic convergence rate, although its performance can be slowed down by sensitivity to the initial guess and the need to recompute the Jacobian at each step. Picard iterations offer a simpler alternative, but often at the cost of a slower convergence. To improve efficiency, specialized iterative solvers have been developed, such as the Peaceman--Rachford method, analyzed in \cite{gir:whe:2008} to solve the resulting discretized nonlinear system. Essentially, it is based on calculating an intermediate solution for a decoupled nonlinear problem, and then solving a linear saddle point system. In addition, this iterative method is used in \cite{huang:chen:rui:2018} as a smoother for a full approximation scheme (FAS), which is the nonlinear version of a multigrid method. Another interesting approach, based on a two-grid method, is proposed in \cite{rui:liu:2015}. Such a method combines the solution of a nonlinear problem on a coarse grid and a linear problem on a target fine grid.

In this paper, we focus on the solution of the nonlinear systems stemming, at each time step, from the space and time discretizations of a slightly compressible Darcy-Forchheimer model. We propose a general iterative linearization scheme based on the $L$-scheme (a stabilized Picard linearization method \cite{listradu2016, pop2004}), which can be applied, in principle, to any spatial discretization. The equations are stabilized by using a free parameter to be chosen, as it is often done when the $L$-scheme is considered. The convergence of the scheme is rigorously shown, under realistic assumptions and provided that the stabilization parameter is large enough. Further, it is studied numerically, first on a series of academic examples and then on more realistic scenarios. A comparison with existing linearization schemes, such as Picard, relaxed Picard and Newton methods, is presented as well. In the Appendix, we also propose and analyze a similar iterative scheme for a simplified version of the Darcy-Forchheimer model, for which the density is assumed to be constant. Note that this variant was previously studied in \cite{al2024iterative,sayah2021} using a slightly different approach.

In summary, the main contributions of this paper are
\begin{itemize}
\item[(a)]{a general iterative linearization scheme is proposed for the slightly compressible Darcy-Forchheimer flow model in porous media;}
\item[(b)]{the theoretical proof for the convergence of such a scheme is provided;}
\item[(c)]{a study of the performance of different linearization schemes (Newton, Picard, relaxed Picard and $L$-scheme) for the proposed model is conducted.}
\end{itemize}

The paper is structured as follows. In Section \ref{sec:mathmodel}, we introduce the mathematical model along with the main notations used throughout the work. The iterative scheme for the slightly compressible Darcy-Forchheimer model is presented in Section \ref{sec:linschemes}, where its convergence is also proven. Section \ref{sec:numerics:1} presents some illustrative numerical experiments, including standard academic benchmarks and more demanding Darcy–Forchheimer flow configurations with discontinuous permeabilities. Finally, Section \ref{sec:conclusion} establishes some concluding remarks. The paper is completed with an Appendix that contains the convergence analysis of a linearization scheme for a simplified version of the Darcy-Forchheimer model.

\section{Mathematical model and discretization} \label{sec:mathmodel}

In this paper, we use standard notations from functional analysis. Let $\Omega \subset \mathbb{R}^d$ be an open, bounded domain with a Lipschitz continuous boundary. Here, $d$ stands for the spatial dimension. Let $T>0$ denote the final time. Given $N \in \mathbb{N}$ with $N\ge1$, $\tau := T/N$ is the time step size and $t_n := n\tau$, for $n = 1, 2, \ldots, N$, denotes the time nodes. For $p \ge 1 $, we denote by $L^p(\Omega)$ the usual Lebesgue space with the standard norm
$$
\| f \|_p := \left(\int_\Omega |f(x)|^p\,dx\right)^{1/p},
$$
for $p < \infty$ and $\| f \|_\infty := ess\,sup_{x \in \Omega} |f(x)| $.

Next, let us describe the slightly compressible flow model considered in this work. The general equation for mass conservation can be written as
\begin{equation}\label{mass:conserv:eq}\partial_t(\phi\rho)+\nabla\cdot(\rho\mathbf{u})=0,\end{equation}
where $\phi$ is the porosity of the medium, and $\mathbf{u}$ and $\rho$ are the velocity and the density of the fluid, respectively. Under isothermal conditions, the state equation relates the density $\rho$ with the pressure $p$, as follows
\begin{equation}\label{state:eq}\frac{1}{\rho}\,\frac{d \rho}{d p}=c_f,\end{equation}
where $c_f$ is the compressibility coefficient. From this expression, it is straightforward to deduce that the fluid density depends on the pressure in the following way 
\[\rho(p)=\rho_{\rm{ref}}\,e^{c_f(p-p_{\rm{ref}})},\]
where $\rho_{\rm{ref}}$ is the reference density at the reference pressure $p_{\rm{ref}}$. For slightly compressible fluids, $c_f$ varies between $10^{-5}$ and $10^{-6}$. Thus, carrying out differentiation in (\ref{mass:conserv:eq}), dividing by $\rho$ and using the relation (\ref{state:eq}), we obtain
\[\phi\,c_f\partial_t p+c_f\nabla p\cdot\mathbf{u}+\nabla\cdot\mathbf{u}=0.\]
For slightly compressible fluids, the second term in the equation above is small in almost all of the domain and can be neglected (cf. \cite{dou:rob:1983,aul:blo:hoa:ibra:2009,rui:pan:2017} and references therein). Finally, considering a rescaling in the time variable and a possible source/sink term $f$, we can express the mass conservation equation as 
\[\partial_t\,p+\nabla\cdot\mathbf{u}=f.\]
On the other hand, the Darcy-Forchheimer equation, which accounts for inertial effects that arise at relatively high flow velocities, is given by
\[\mu K^{-1}\mathbf{u}+\beta\rho\,|\mathbf{u}|\,\mathbf{u}+\nabla p=\rho g\nabla Z,\]
where $\mu$ denotes the fluid viscosity, $K$ is the permeability tensor, $\beta\geq 0$ represents the Forchheimer coefficient, $g$ is the gravitational constant, and $Z$ denotes the depth of the domain. A theoretical derivation of Forchheimer’s law can be found in \cite{rut:ma:1992}. In particular, when $\beta=0$, the nonlinear Darcy-Forchheimer equation reduces to the linear Darcy's law.

Henceforth, we consider an isotropic and homogeneous medium, so that the permeability tensor $K=kI_d$, where $I_d$ denotes the identity matrix, reduces to a constant scalar value $k$. We also assume that the depth function $Z$ is constant, and therefore $\nabla Z$ vanishes. Consequently, the evolution of a slightly compressible Darcy-Forchheimer flow is governed by the following system of equations
\begin{subequations}\label{ibvp}
	\begin{align}
	\mu k^{-1}\mathbf{u}+\beta\rho(p)\,|\mathbf{u}|\,\mathbf{u}+\nabla p&=\mathbf{0}&&\hspace*{-1.5cm}\mbox{in }\Omega\times(0,T],\label{ibvp:a}\\[0.5ex]
	 \partial_tp+\nabla\cdot\mathbf{u}&=f&&\hspace*{-1.5cm}\mbox{in }\Omega\times(0,T],\label{ibvp:b}
	\end{align}
\end{subequations}
together with suitable initial and boundary conditions. After discretizing in time by using the backward Euler method, we have to solve the following problem at each time step: \emph{Find $(\mathbf{u}^n, p^n)$ such that}
\begin{subequations}\label{bvp}
	\begin{align}
	\mu k^{-1}\mathbf{u}^n+\beta\rho(p^n)\,|\mathbf{u}^n|\,\mathbf{u}^n+\nabla p^n&=\mathbf{0}&&\hspace*{0cm}\mbox{in }\Omega,\label{bvp:a}\\[0.5ex]
	 p^n+\tau\nabla\cdot\mathbf{u}^n&=p^{n-1}+\tau f^n&&\hspace*{0cm}\mbox{in }\Omega,\label{bvp:b}
	\end{align}
\end{subequations}
where, as mentioned before, $\tau$ denotes the time step size and $n$ denotes the time level. When considering homogeneous Dirichlet boundary conditions, the variational formulation of the preceding system reads as follows: \emph{Find $(\mathbf{u}^n, p^n)\in V\times W$ such that, for all $\bv\in V$ and $q\in W$, it holds} 
\begin{subequations}\label{weak:form}
\begin{align}
\la\mu k^{-1}\mathbf{u}^n,\mathbf{v} \ra+\beta\la\rho(p^n)\,|\mathbf{u}^n|\,\mathbf{u}^n,\mathbf{v} \ra-\la p^n,\nabla\cdot\mathbf{v} \ra&=0,&&\label{eq:discrete_u}\\[1ex]
 \la \,p^n,q \ra+\tau\,\la \nabla\cdot\mathbf{u}^n,q \ra&=  \la p^{n-1},q \ra+\tau \la f^n,q \ra,&&\label{eq:discrete_p}
\end{align}
\end{subequations}
where $\la\cdot,\cdot\ra$ denotes the inner product in $L^2(\Omega)$ or $(L^2(\Omega))^d$, and the function spaces are defined to be (cf. \cite{pan:rui:2012,xu:lia:rui:2017})
\[V=\{\mathbf{v}\in (L^3(\Omega))^d:\nabla\cdot\mathbf{v}\in L^2(\Omega)\},\qquad W=L^{2}({\Omega}).
\]

\begin{remark} For the existence and uniqueness of solution of the system \eqref{weak:form}, we refer to \cite{kieu:2020}, where the author considers a similar weak formulation for the semidiscrete slightly compressible generalized Forchheimer equations with the density and the momentum as unknowns.
\end{remark}

\section{Linearization schemes}\label{sec:linschemes}

In the sequel, we introduce a class of linearization schemes for the system \eqref{weak:form}, depending on a parameter $ \gamma \in [0,1]$. Throughout the paper, $i$ denotes the iteration index, while $n$ stands for the time level. Such schemes can be applied to any spatial discretization of the system, and the presentation could be equivalently formulated for the fully discrete variational formulation. Nevertheless, for the sake of clarity, we present them for the semi-discrete variational setting provided by \eqref{eq:discrete_u} and \eqref{eq:discrete_p}, thus avoiding the introduction of an additional index $h$ connected to the spatial discretization spaces.

In particular, in Section \ref{sec:numerics:1}, the linearization schemes are implemented using a spatial discretization based on the two-point flux approximation (TPFA) method. This discretization is equivalent to the lowest order Raviart--Thomas MFE formulation when combined with an appropriate quadrature rule.

{\bf Linearization scheme ($L$-scheme $ \gamma$).} Let $L \ge 0$ be arbitrary, $ \gamma \in [0,1]$, $ {\bf u}^{n,0} := {\bf u}^{n-1}$, $p^{n,0} := p^{n-1}$ and $i \ge 1$. Given $({\bf u}^{n, i-1}, p^{n,i-1}) \in V\times W$, find $(\ui, p^{n,i})\in V\times W$ such that, for all $\bv \in V$ and $q \in W$, it holds
\begin{subequations}
\begin{align}
    \la\mu k^{-1} \ui, \bv \ra + \beta \la \rho(p^{n,i-1})\,|\,{\bf u}^{n,i-1}|\, (\gamma {\bf u}^{n,i-1} &+ (1-\gamma) \ui), \bv \ra \nonumber\\
    + L \la \ui - {\bf u}^{n,i-1}, \bv \ra - \la p^{n,i}, \nabla \cdot \bv \ra &= 0,\label{eq:linscheme:1}\\[1ex]
 \la p^{n,i}, q \ra + \tau \la \nabla \cdot \ui, q \ra &=  \la p^{n,0}, q \ra+\tau\la f^n, q \ra.
\label{eq:linscheme:2}
\end{align}
\end{subequations}

Throughout this work, the schemes corresponding to $ \gamma = 0 $ and $ \gamma = 1 $ in \eqref{eq:linscheme:1}-\eqref{eq:linscheme:2} are studied in detail, since they employ the most common linearization methods for the term $\rho(p^{n})\,|{\bf u}^{n}|\,{\bf u}^{n}$. 

In the sequel, we will study the theoretical convergence of the $L$-scheme $\gamma$. We refer to the \ref{sec:appendix} for the convergence of the iterative scheme in a simplified Darcy-Forchheimer model, when $\rho \equiv 1$.

Throughout this paper, we will make use of the following assumptions:

 \begin{itemize}
    \item[(A1)] The permeability constant and the Forchheimer number satisfy  $k >0$ and $\beta > 0$, respectively.
  \item[(A2)]  For any $n\in \mathbb{N}$, there exists $M_u \in (0, +\infty) $ such that $ \| \ui \|_{L^{\infty}(\Omega)} \le M_u $, for all $i$. Note that $M_u$ might depend on the time step level $n$.
  \item[(A3)] The function $\rho : \mathbb{R} \rightarrow \mathbb{R}$ is Lipschitz continuous, with Lipschitz constant $L_\rho$, monotonically increasing and satisfies
  \begin{equation*}
     0 < m_\rho \le \rho(x) \le M_\rho < \infty \quad \forall\,x \in \mathbb{R}.
  \end{equation*}
\end{itemize}

\begin{remark} The assumption (A1) can be easily relaxed. For example, it can be assumed that the permeability tensor $K$ is a $d\times d$ symmetric and positive definite tensor that satisfies 
\begin{equation*}
{\kappa}_{\ast}\,\mathbf{\xi}^T\mathbf{\xi}\leq
\mathbf{\xi}^{T}{K}\,\mathbf{\xi}\leq
{\kappa}^{\ast}\,\mathbf{\xi}^T\mathbf{\xi}
\quad\forall\,\mathbf{\xi}\neq\mathbf{0}\in\mathbb{R}^d,
\end{equation*}
for some $0<{\kappa}_{\ast}\leq{\kappa}^{\ast}<\infty$. In the same spirit, one can assume that $\beta \in L^{\infty}(\Omega) $ and satisfies, for all $x \in \Omega $,
\begin{eqnarray*}
     0 < \beta_m \le \beta (x) \le \beta_M < \infty,
\end{eqnarray*}
with $\beta_m$, $\beta_M\in\mathbb{R}.$
\end{remark}

\begin{remark} The assumption (A2) is realistic for smooth data. It can be verified by showing the boundedness of the fluxes in smoother spaces than $V$.
\end{remark}

\begin{remark} The assumption (A3) on the density function is satisfied for realistic data. In particular, it is satisfied for the numerical examples used in Section \ref{sec:numerics:1}. 
\end{remark}

In what follows, we will use the following result (see \cite{kna:rob:2014, summ_thesis}):
\begin{lemma} \label{lemma_inequalities} For all $ \mathbf{x},\,\mathbf{y} \in \mathbb{R}^n$, it holds 
\begin{eqnarray}
\la |\mathbf{x}|\,\mathbf{x} - |\mathbf{y}|\,\mathbf{y}, \mathbf{x} - \mathbf{y} \ra \ge \dfrac{1}{2}\,| \mathbf{x}-\mathbf{y} |^{3}. \label{lemma:ineq:2} 
\end{eqnarray} 
\end{lemma}
We can now state the main result regarding the convergence of the proposed linearization schemes.
\begin{theorem} \label{theorem_conv_Lscheme1} Assuming that (A1)-(A3) hold true, the $L$-scheme $\gamma$ \eqref{eq:linscheme:1}-\eqref{eq:linscheme:2} converges, for a small enough time step $\tau$, if the parameter {\it L} satisfies
\begin{equation} \label{convergence_condition}
    L > \dfrac{2k\beta^2M_\rho^2M_u^2\,(1 + \gamma^2) }{\mu}.
\end{equation}
\end{theorem}
{\bf Proof.} We begin by introducing the errors at the $i$th iteration, i.e.,
\begin{align*}
 \eiu &:= {\bu}^{n,i} - {\bu}^n, \\[0ex]
 \epi &:= p^{n,i} - p^n. 
\end{align*}
By subtracting the equations \eqref{eq:discrete_u} and \eqref{eq:discrete_p} from \eqref{eq:linscheme:1} and \eqref{eq:linscheme:2}, respectively, we get
\begin{displaymath}
\begin{array}{l}
   \langle \mu k^{-1} \eiu, \bv \rangle  + \beta\,\la \rho(p^{n,i-1}) \,|{\bf u}^{n, i-1}|\, (\gamma {\bf u}^{n,i-1} + (1-\gamma)\,\ui) - \rho(p^{n})\,|{\bf u}^{n}|\,{\bf u}^{n}, \bv \rangle \nonumber\\[1ex]
  \hspace*{2cm} +\,L\,\la \eiu - \fui, \bv \rangle -  \langle \epi, \nabla \cdot \bv \ra  = 0 
\end{array}
\end{displaymath}
and
\begin{equation*}
 \la \epi, q \rangle + \tau\,\la \nabla \cdot \eiu, q \ra = 0, 
\end{equation*}
  for all $\bv \in V$ and $q \in W$.
Now, we test the above expressions with $\bv = \tau \eiu$ and $ q = \epi $, respectively, and add the resulting equations to obtain
\begin{equation}\label{eq:proof:1}
\begin{array}{l}
 \tau\mu k^{-1} \| \eiu \|^{2} + \tau \beta\,\la \rho(p^{n,i-1})\,|{\bf u}^{n, i-1}|\,\gamma\,({\bf u}^{n, i-1} - {\bf u}^{n, i}) , \eiu \rangle \nonumber \\[1ex]
\hspace*{2cm} +\,\tau\beta\,\la \rho(p^{n,i-1})\,|{\bf u}^{n, i-1}|\,{\bf u}^{n, i} - \rho(p^{n})\,|{\bf u}^{n}|\,{\bf u}^{n}, \eiu \rangle \\[1ex] 
\hspace*{2cm} +\,\tau L\,\langle \eiu - \fui, \eiu \rangle +  \| \epi \|^{2} = 0.
       \end{array}
\end{equation}
This is further equivalent to 
\begin{equation}\label{eq:proof:2}
\begin{array}{l}
    \tau\mu k^{-1} \| \eiu \|^{2} + \tau \beta \gamma\,\la \rho(p^{n,i-1})\,|{\bf u}^{n, i-1}|\,(\fui-\eiu), \eiu \rangle  \\[1ex]
\hspace*{2cm}     + \, \tau \beta\,\la \rho(p^{n,i-1})\,|{\bf u}^{n, i-1}|\,{\bf u}^{n, i} - \rho(p^{n})\,|{\bf u}^{n}|\,{\bf u}^{n}, \eiu \rangle   \\[1ex]
  \hspace*{2cm}  + \, \dfrac{\tau L }{2}\,\| \eiu \|^{2}  + \dfrac{\tau L }{2}\,\| \eiu - \fui \|^{2} +   \| \epi \|^{2} = \dfrac{\tau L }{2}\,\| \fui \|^{2}.
    \end{array}    
\end{equation}
By denoting
\begin{align*}
    T_1 &:=\tau \beta \gamma\,\la \rho(p^{n,i-1})\,|{\bf u}^{n, i-1}|\,(\fui-\eiu), \eiu \rangle, \\[0ex]
    T_2 &:=\tau \beta\, \la \rho(p^{n,i-1})\,|{\bf u}^{n, i-1}|\,{\bf u}^{n, i} - \rho(p^{n})\,|{\bf u}^{n}|\,{\bf u}^{n}, \eiu \rangle, 
\end{align*}
we can rewrite \eqref{eq:proof:2} as
\begin{equation} \label{eq:proof:3}
   \left(\tau\mu k^{-1} + \dfrac{\tau L }{2}\right) \| \eiu \|^{2} + T_1 + T_2 + \dfrac{\tau L }{2}\,\| \eiu - \fui \|^{2} +   \| \epi \|^{2} = \dfrac{\tau L }{2}\,\| \fui \|^{2}.  
\end{equation}
Next, we will look at the terms $T_1$ and $T_2$ separately. By using (A3) and the Cauchy-Schwarz and Young inequalities, we get
\begin{align}
 T_1 &= \tau \beta \gamma\, \la \rho(p^{n,i-1})\,|{\bf u}^{n, i-1}|\,(\fui-\eiu), \eiu \rangle \nonumber \\[1ex]
 & \leq \tau \beta \gamma M_\rho M_u \,\| \eiu - \fui \|\, \| \eiu \| \nonumber\\[1ex]
 & \leq  \dfrac{\tau C_1}{L}\, \| \eiu \|^2 +  \dfrac{\tau L}{4}\, \| \eiu - \fui \|^2, \label{eq:proof:4}
\end{align}
where $C_1 := (\beta \gamma M_\rho M_u)^2$. For the term $T_2$, it holds
\begin{align}
  T_2 &= \tau \beta\,\la \rho(p^{n,i-1})\, |{\bf u}^{n, i-1}|\, {\bf u}^{n, i} - \rho(p^{n})\, |{\bf u}^{n}|\, {\bf u}^{n}, \eiu \rangle   \nonumber \\[1ex]
 & = \tau \beta\, \la (\rho(p^{n,i-1}) - \rho(p^{n})) \,|{\bf u}^{n, i-1}| \,{\bf u}^{n, i}, \eiu \rangle  \nonumber\\[1ex]
 &\hspace*{2cm}+ \tau \beta\, \la \rho(p^{n})\, (|{\bf u}^{n, i-1}|\, {\bf u}^{n, i} -  |{\bf u}^{n}|\, {\bf u}^{n}), \eiu \rangle \nonumber\\[1ex]
 & = \tau \beta\, \la (\rho(p^{n,i-1}) - \rho(p^{n}))\, |{\bf u}^{n, i-1}|\, {\bf u}^{n, i}, \eiu \rangle  \nonumber\\[1ex]
 &\hspace*{2cm}+ \tau \beta\, \la \rho(p^{n})\, (|{\bf u}^{n, i-1}| - |{\bf u}^{n, i}|)\,{\bf u}^{n, i}, \eiu \rangle \nonumber\\[1ex]
 &\hspace*{2cm}+ \tau \beta \,\la \rho(p^{n})\, (|{\bf u}^{n, i}|\, {\bf u}^{n, i} -  |{\bf u}^{n}|\, {\bf u}^{n}), \eiu \rangle \nonumber\\[1ex]
& = T_{21}  +  T_{22}  + T_{23}. \label{eq:proof:5}
\end{align}
We continue by estimating the terms $T_{21}, T_{22}, T_{23}$ above. Using the Lipschitz continuity of $\rho(\cdot)$, the assumption (A2) and the Cauchy-Schwarz and Young inequalities, it follows that
\begin{align}
 T_{21} & =  \tau \beta\,\la (\rho(p^{n,i-1}) - \rho(p^{n}))\, |{\bf u}^{n, i-1}|\, {\bf u}^{n, i}, \eiu \rangle  \nonumber\\[1ex]
& \le \tau \beta M_u^2 L_\rho\, \| \eiu \|\, \| \eip1 \|\nonumber\\[1ex]
& \le \dfrac{\tau\mu k^{-1}}{2}\,\| \eiu \|^{2} + \tau C_{21}\, \| \eip1 \|^2, \label{eq:proof:6}
\end{align}
with $C_{21} := \beta^2 M_u^4 L_\rho^2 k/(2\mu)$. For the next term, using the assumptions (A2) and (A3), together with the Cauchy-Schwarz and Young inequalities, we get
\begin{align}
 T_{22} & =  \tau \beta \,\la \rho(p^{n})\, (|{\bf u}^{n, i-1}| - |{\bf u}^{n, i}|)\, {\bf u}^{n, i}, \eiu \rangle   \nonumber\\[1ex]
& \le  \tau \beta  M_\rho M_u\, \| \eiu - \fui \|\, \| \eiu \|  \nonumber\\[1ex]
& \le  \dfrac{ \tau L}{4}\, \| \eiu - \fui \|^2 + \dfrac{\tau C_{22}}{L}\, \| \eiu \|^2, \label{eq:proof:7}
\end{align}
with $C_{22} := \beta^2 M_\rho^2 M_u^2$. Finally, by (A3) and the inequality \eqref{lemma:ineq:2},  the term $T_{23}$ can be bounded from below as follows
\begin{align}
 T_{23} & =  \tau \beta\, \la \rho(p^{n})\, (|{\bf u}^{n, i}|\, {\bf u}^{n, i} -  |{\bf u}^{n}|\, {\bf u}^{n}), \eiu \rangle  \nonumber\\[1ex]
& \ge \dfrac{\tau \beta m_\rho}{2}\, \| \eiu \|_{L^3(\Omega)}^3 \ge 0. \label{eq:proof:8}
\end{align}
Using now \eqref{eq:proof:4}-\eqref{eq:proof:8} in \eqref{eq:proof:3}, we get
\begin{equation} \label{eq:proof:9}
\begin{array}{l}
   \left(\dfrac{\tau \mu k^{-1}}{2} + \dfrac{\tau L }{2}  - \dfrac{\tau(C_1 + C_{22})}{L}\right) \| \eiu \|^{2}
   +  \| \epi \|^{2} + \dfrac{\tau \beta m_\rho}{2}\, \| \eiu \|_{L^3(\Omega)}^3 \\[1ex] 
   \hspace*{2cm}\le \dfrac{\tau L }{2}\,\| \fui \|^{2} + \tau C_{21}\, \| \eip1 \|^2.  
\end{array} 
\end{equation}
The preceding expression provides a contraction if 
\begin{equation} \label{eq:proof:10}
  \dfrac{\tau \mu\,k^{-1}}{2} >  \dfrac{\tau(C_1 + C_{22})}{L} \end{equation}
and
\begin{equation} \label{eq:proof:11}
  1  >  \tau\,C_{21}.  
\end{equation}
We remark that \eqref{eq:proof:10} is equivalent to \eqref{convergence_condition}, and \eqref{eq:proof:11} is satisfied when $\tau$ is small enough. This completes the proof.\\[1ex]
\mbox{ }\hfill {\bf Q.E.D.} 

\begin{remark}
 In the case $\gamma = 0$, the scheme \eqref{eq:linscheme:1}-\eqref{eq:linscheme:2} converges for the largest range of $L$.  This is intuitively justified, because the  choice $\gamma = 0$ corresponds to a linearization of the term $\rho(p^n)\,|\mathbf{u}^n|\,\mathbf{u}^n $ that involves not only $\mathbf{u}^n$ from the last iteration, but also from the current one. It is also consistent with the numerical results, where the $L$-scheme with $\gamma = 0$ performs always better than that considering $ \gamma = 1$.
\end{remark}

\begin{remark}
The positive term $\tau\beta m_\rho\,\| \eiu \|_{L^3(\Omega)}^3/2 $ in \eqref{eq:proof:9} was not used in the contraction estimates. This is a hint that the convergence properties of the scheme might be better.
\end{remark}

\section{Numerical experiments}\label{sec:numerics:1}
In this section, we are going to illustrate the behaviour of the proposed linearization scheme \eqref{eq:linscheme:1}-\eqref{eq:linscheme:2}. Moreover, we shall compare its performance with that of three classical linearization schemes, namely the Newton method, the Picard method, and the relaxed Picard scheme. 

The Picard method corresponds to scheme \eqref{eq:linscheme:1}-\eqref{eq:linscheme:2} considering $\gamma=L=0$.  This scheme was the solver considered in \cite{pan:rui:2012}, where a discretization of mixed finite elements was analyzed for the incompressible Darcy-Forchheimer model. In that paper, the convergence of the iterative algorithm was tested on different examples, considering a value $\beta=1$, where it performed quite well. As we shall see next, as the Forchheimer parameter $\beta$ grows, especially if $k$ increases as well, the convergence of the algorithm deteriorates drastically. In those cases, we can consider the relaxed Picard method, which can be defined as follows: once $(\mathbf{u}^{n,i},p^{n,i})$ are computed with scheme 
\eqref{eq:linscheme:1}-\eqref{eq:linscheme:2} considering $\gamma=L=0$, we update their value with a relaxation parameter $\omega$ as follows:  $(\mathbf{u}^{n,i},p^{n,i})\leftarrow \omega(\mathbf{u}^{n,i},p^{n,i})+(1-\omega) (\mathbf{u}^{n,i-1},p^{n,i-1})$. 

All numerical experiments were performed in MATLAB. CPU times were measured on a desktop computer equipped with an Intel Core i5-14500 processor (2.60 GHz) and 24 GB of RAM.

\subsection{A test with known analytical solution}

First, we consider problem \eqref{ibvp}
posed on $\Omega\times(0,T)=(0,1)^2\times(0,1)$ together with homogeneous Dirichlet boundary conditions for the pressure. In particular, we consider
\[p(x,y,t)=e^{-2t}\,x\,(1-x)\,y\,(1-y),\] 
and, according to \cite{aul:blo:hoa:ibra:2009}, we can obtain the exact velocity field $\mathbf{u}$ by imposing the slightly compressible Darcy-Forchheimer equation
\[\left(\dfrac{\mu}{k}+\beta\,\rho(p)\,|\mathbf{u}|\right)\mathbf{u}=-\nabla p.\]
From the previous equation, we can deduce that 
\[\beta\rho(p)\,|\mathbf{u}|^2+\dfrac{\mu}{k}\,|\mathbf{u}|-|\nabla p|=0.\]
Then, as long as $\beta$ is different from zero, since  $|\mathbf{u}|$ is non-negative, we have
\[|\mathbf{u}|=\dfrac{-\dfrac{\mu}{k}+\sqrt{\dfrac{\mu^2}{k^2}+4\beta\rho(p)\,|\nabla p|}}{2\beta\rho(p)}.\]
Finally, substituting in the first equation, we get
\[\mathbf{u}=\dfrac{-2\nabla p}{\dfrac{\mu}{k}+\sqrt{\dfrac{\mu^2}{k^2}+4\beta\rho(p)\,|\nabla p|}}.\]
Moreover, we consider the following values for the rest of the parameters of the problem: $\mu=\rho_{\rm{ref}}=1$, $p_{\rm{ref}}=0$ and $c_f=10^{-5}$.

In order to approximate the solution of the previous problem, we combine the backward Euler method with time step $\tau$ and the TPFA scheme on a Cartesian staggered grid with mesh size $h$. The resulting nonlinear system arising at each time step is iteratively solved with the aforementioned linearization schemes, considering the following stopping criterion for the difference between the solution of two consecutive iterations
\begin{equation}\label{stopping:crit}\|V^{n,j}-V^{n,j-1}\|\leq \mathtt{tol}_a+\mathtt{tol}_r\,\|V^{n,j}\|,\end{equation} with $\mathtt{tol}_a=\mathtt{tol}_r=10^{-5}$, where $\|\cdot\|$ denotes the discrete $\ell_2$-norm. In particular, we compare the number of iterations, the required computational time, and the condition numbers of the associated linear systems for each of the linearization methods. 

	\begin{table}[t]
	\small
	\renewcommand{\arraystretch}{1.2}
	\begin{center}
		\begin{tabular}{|l|c|c|c|c|c|} \cline{2-6}
			\multicolumn{1}{c|}{} & $\tau=1$ & $\tau=0.5$ & $\tau=0.1$ & $\tau=0.01$ & $\tau=0.001$  \\
			\hhline{|------|}
			Av\_It\_N & 6 & 4.5 & 3 & 2 & 2 \\
				Time\_N& 4.3e$-$1 & 6.6e$-$1 & 2.3e+0 & 1.6e+1 & 1.6e+2  \\\hline
				Av\_It\_P & 3 & 3.5 & 3.4 & 2.9 & 2.2 \\
				Time\_P& 1.1e$-$1 & 2.3e$-$1 & 1.1e+0 & 1.0e+1 & 8.2e+1 
                \\\hline
                Av\_It\_RP & 10 & 9 & 6.7 & 4.4 &  2.6\\
			Time\_RP& 2.8e$-$1 & 5.1e$-$1 & 2.0e+0& 1.4e+1 & 9.3e+1\\
				\hline
				Av\_It\_L0 & 5 & 4 & 3.2 & 2.3 & 2 \\
			Time\_L0& 1.6e$-$1 & 2.6e$-$1 & 1.1e+0 & 8.6e+0 & 7.7e+1	\\
                 \hline
 				Av\_It\_L1 & 9 & 7 & 4.5 & 3.9 & 2.4\\
 				Time\_L1& 2.6e$-$1 & 4.2e$-$1 & 1.4e+0 & 1.3e+1 &  8.8e+1 \\
			\hline
		\end{tabular}
		\caption{Average number of iterations per time step and total CPU time (in seconds) when considering the following methods: Newton (N), Picard (P), relaxed Picard (RP) with relaxation coefficient $\omega=0.7$, $L$-scheme with $\gamma=0$ (L0) and $L=0.07$, and  $L$-scheme with $\gamma=1$ (L1) and $L=0.4$. 
        In this case, $h=1/80$, and $k=\beta=1$.}\label{table:influence:time:step:tpfa}
	\end{center}
\end{table}

Table \ref{table:influence:time:step:tpfa} shows the the average number of iterations per time step and the total CPU time (in seconds) for the different methods when considering $k=\beta=1$, $h=\frac{1}{80}$ and the following values for $\tau\in\{1,0.5,0.1,0.01,0.001\}.$ Here and henceforth, the CPU times have been computed using the MATLAB command \texttt{timeit}, which executes the code multiple times and reports the median runtime. As the time step size decreases, the Newton method yields the smallest average number of iterations per time step, followed by the $L$-scheme (with $\gamma=0$), the classical Picard method, the $L$-scheme (with $\gamma=1$) and the relaxed Picard scheme. However, in terms of CPU time, the schemes were very similar except for the Newton method which was the slowest one. This was especially due to the linear systems to be solved within the Newton iterations being very bad conditioned.

Next, we illustrate how the mesh size $h$ and the parameter values $\beta$ and $k$ influence the number of iterations and the CPU times. In all cases, we compute one time step of size $\tau=1$ and we consider different mesh sizes $h\in\{1/40, 1/80, 1/160, 1/320\}$. The values considered for the parameters are $\beta\in\{1,10,100\}$ and $k\in\{0.01,1,100\}$. 

Among the family of $L$-schemes proposed in the paper, we consider the one with $\gamma=0$, as it yields the best performance in terms of both iteration count and CPU time.
Parameter $L$ for the $L$-scheme is only varied with respect to $\beta$. In fact, it is considered to be proportional to $\beta^{1/2}$. In turn, the Picard relaxation parameter is considered to be $\omega=0.7$.

	\begin{center}
		\begin{table}[t]
			\begin{center}
				\begin{tabular}{|c|c|c|c|c|c|c|c|c|c|c|}\cline{3-10}
					\multicolumn{2}{c|}{}& \multicolumn{2}{c|}{Newton} &  \multicolumn{2}{c|}{Picard}&  \multicolumn{2}{c|}{Rel. Picard} &  \multicolumn{2}{c|}{$L$-scheme}\\\hline
					$N$&$k$& It &  Time & It  & Time &It  & Time &It  & Time \\\hline
					\multirow{3}{*}{40} 
					&$10^{-2}$  & 3 & 5.1e$-$2 & 3 & 2.9e$-$2 & 11 & 7.8e$-$2 & 3 &  2.7e$-$2  \\\cline{2-10}
					&1  & 6 & 8.9e$-$2 & 3 & 2.7e$-$2  & 10 & 7.2e$-$2 & 5 & 4.0e$-$2 \\\cline{2-10}
					&$10^{2}$  & 9 & 1.3e$-$1 & 73 & 4.7e$-$1 & 11 & 8.0e$-$2 & 11 &  7.8e$-$2  \\\hline
					\multirow{3}{*}{80} 
					&$10^{-2}$  & 3 & 2.2e$-$1 & 3 & 1.1e$-$1 & 11 & 3.1e$-$1 & 3 & 1.1e$-$1 \\\cline{2-10}
					&1  & 6 & 4.2e$-$1 & 3 & 1.1e$-$1 & 10 & 2.9e$-$1 &  5 & 1.6e$-$1  \\\cline{2-10}
					&$10^{2}$  & 9 & 6.1e$-$1 & 79 & 2.2e+0 & 11 & 3.3e$-$1 & 11 & 3.1e$-$1 \\\hline
					\multirow{3}{*}{160} 
					&$10^{-2}$  & 3 & 1.6e+0 & 3 & 4.5e$-$1 & 11 & 1.4e+0 & 3 & 4.5e$-$1 \\\cline{2-10}
					&1  & 6 & 3.1e+0 & 3 & 4.5e$-$1 & 10 & 1.3e+0 & 5 & 6.9e$-$1 \\\cline{2-10}
					&$10^{2}$  & 10 & 5.2e+0 & 82 & 1.1e+1 & 11 & 1.7e+0 & 11 & 1.4e+0 \\\hline
					\multirow{3}{*}{320} 
					&$10^{-2}$  & 3 & 1.8e+1 & 3 & 1.9e+0 & 11 & 6.6e+0 & 3 & 1.9e+0 \\\cline{2-10}
					&1 & 6 & 3.6e+1 & 3 & 1.9e+0 & 10 & 5.9e+0 & 5 & 3.0e+0 \\\cline{2-10}
				    &$10^{2}$   & 10 & 5.9e+1 & 83 & 8.2e+2 &  11 & 8.4e+1  & 11 & 1.3e+1 \\\hline
				\end{tabular}
			\end{center}
			\caption{One time step $\tau=1$, $\beta=1$, $c_f=10^{-5}$, Picard relaxation parameter $\omega=0.7$, $L$-scheme ($\gamma=0$) parameter $L=0.07$.}\label{table:beta:1}
		\end{table}
	
		\begin{table}[h!]
			\begin{center}
				\begin{tabular}{|c|c|c|c|c|c|c|c|c|c|c|}\cline{3-10}
					\multicolumn{2}{c|}{}& \multicolumn{2}{c|}{Newton} &  \multicolumn{2}{c|}{Picard}&  \multicolumn{2}{c|}{Rel. Picard} &  \multicolumn{2}{c|}{$L$-scheme}\\\hline
					$N$&$k$& It &  Time & It  & Time &It  & Time &It  & Time\\\hline
					\multirow{3}{*}{40} 
					&$10^{-2}$  & 4 & 6.3e$-$2 & 3 & 2.7e$-$2 & 11 & 7.7e$-$2 & 3 &  2.7e$-$2\\\cline{2-10}
					&1  & 8 & 1.1e$-$1 & 5 & 4.0e$-$2 &  10 & 7.3e$-$2 & 6 & 4.6e$-$2\\\cline{2-10}
					&$10^{2}$  & 11 & 1.5e$-$1  & 132 & 8.5e$-$1 & 11 & 7.9e$-$2 & 9 & 6.6e$-$2 \\\hline
					\multirow{3}{*}{80} 
					&$10^{-2}$  & 4 & 2.8e$-$1  & 3 & 1.0e$-$1 & 11 & 3.1e$-$1 & 3 & 1.0e$-$1 \\\cline{2-10}
					&1  & 8 & 5.4e$-$1  & 5 & 1.6e$-$1 & 10 & 2.9e$-$1 & 6 & 1.8e$-$1 \\\cline{2-10}
					&$10^{2}$  & 12 & 8.0e$-$1 & $-$ & $-$ & 11 &  3.2e$-$1& 9 & 2.7e$-$1 \\\hline
					\multirow{3}{*}{160} 
					&$10^{-2}$  & 4 & 2.1e+0  & 3 & 4.5e$-$1 & 11 & 1.4e+0 & 3 & 4.5e$-$1 \\\cline{2-10}
					&1  & 8 & 4.2e+0  & 5 & 6.9e$-$1 & 10 & 1.3e+0 & 6 & 8.0e$-$1 \\\cline{2-10}
					&$10^{2}$  & 12 & 6.2e+0  & $-$ & $-$ & 11 & 1.4e+0 & 9 & 1.2e+0 \\\hline
					\multirow{3}{*}{320} 
					&$10^{-2}$  & 4 & 2.4e+1  & 3 & 1.9e+0 & 11 & 6.5e+0 & 3 & 1.9e+0 \\\cline{2-10}
					&1  & 8 & 4.8e+1  & 5 & 3.1e+0 & 10 & 6.0e+0& 6 & 3.6e+0 \\\cline{2-10}
					&$10^{2}$  & 12 & 7.1e+1  & $-$ & $-$ & 11 & 7.7e+0 & 9 & 5.3e+0  \\\hline
					\end{tabular}
			\end{center}
			\caption{One time step $\tau=1$, $\beta=10$, $c_f=10^{-5}$, Picard relaxation parameter $\omega=0.7$, $L$-scheme ($\gamma=0$) parameter $L=0.22$.}\label{table:beta:1e+1}
		\end{table}
	
		\begin{table}[h!]
			\begin{center}
				\begin{tabular}{|c|c|c|c|c|c|c|c|c|c|c|}\cline{3-10}
					\multicolumn{2}{c|}{}& \multicolumn{2}{c|}{Newton} &  \multicolumn{2}{c|}{Picard}&  \multicolumn{2}{c|}{Rel. Picard} &  \multicolumn{2}{c|}{$L$-scheme}\\\hline
					$N$&$k$& It &  Time & It  & Time &It  & Time &It  & Time\\\hline
					\multirow{3}{*}{40} 
					&$10^{-2}$  & 6 & 9.0e$-$2 & 3 & 2.7e$-$2 & 11 & 7.8e$-$2 & 4 &  3.3e$-$2 \\\cline{2-10}
					&1  & 11 & 1.5e$-$1  &  10 & 7.1e$-$2  & 10 & 7.2e$-$2 & 6 & 4.5e$-$2\\\cline{2-10}
					&$10^{2}$  & 12 & 1.7e$-$1 & $-$ & $-$ & 11 & 7.7e$-$2 & 9 & 6.4e$-$2 \\\hline
					\multirow{3}{*}{80} 
					&$10^{-2}$  & 6 & 4.2e$-$1  & 3 & 1.1e$-$1 & 11 & 3.2e$-$1 & 4 & 1.3e$-$1 \\\cline{2-10}
					&1  & 11 &  7.4e$-$1  & 10 & 2.9e$-$1 & 10 & 2.9e$-$1 & 6 & 1.8e$-$1 \\\cline{2-10}
					&$10^{2}$  & 12 & 8.0e$-$1  & $-$ & $-$ & 11 & 3.1e$-$1 & 9 & 2.6e$-$1 \\\hline
					\multirow{3}{*}{160} 
					&$10^{-2}$  & 6 & 3.1e+0  & 3 & 4.5e$-$1 & 11 & 1.4e+0 & 4 & 5.7e$-$1 \\\cline{2-10}
					&1  & 11 & 5.7e+0  & 10 & 1.3e+0 & 10 & 1.3e+0 & 6 & 8.1e$-$1 \\\cline{2-10}
					&$10^{2}$  & 13 & 6.7e+0  & $-$ & $-$ & 12 & 1.5e+0 & 9 & 1.2e+0 \\\hline
					\multirow{3}{*}{320} 
					&$10^{-2}$  & 6 & 3.6e+1  & 3 & 1.9e+0 & 11 & 6.5e+0 & 4 &  2.5e+0 \\\cline{2-10}
					&1 & 11 & 6.5e+1 & 10 & 5.9e+0 & 10 & 5.9e+0 & 6 & 3.6e+0 \\\cline{2-10}
				    &$10^{2}$   & 13 & 7.7e+1  & $-$ & $-$ & 12 & 7.1e+0& 9 & 5.3e+0  \\\hline
				\end{tabular}
			\end{center}
			\caption{One time step $\tau=1$, $\beta=100$, $c_f=10^{-5}$, Picard relaxation parameter $\omega=0.7$, $L$-scheme ($\gamma=0$) parameter $L=0.7$.}\label{table:beta:1e+2}
		\end{table}
	\end{center}

Tables \ref{table:beta:1}-\ref{table:beta:1e+2} show that all the methods are robust with respect to the spatial mesh size $h$. However, it is important to note that increasing the parameters $k$ and $\beta$ enhance the dominance of the nonlinear term over the linear one. As a consequence, as $k$ and/or $\beta$ increase, the methods require more iterations to achieve convergence. In particular, for sufficiently large values of $\beta$ and $k$, the convergence of the classical Picard method significantly deteriorates. The notation ``$-$'' in the tables indicates that the method failed to converge within the prescribed maximum number of iterations, set at 150.

Regarding computational efficiency, the $L$-scheme consistently outperforms the Newton method in terms of CPU time. Moreover, as $\beta$, and especially $k$, increase, the $L$-scheme exhibits the lowest CPU time among all the methods considered.

Next, we illustrate the influence of the mesh size $h$ and the parameter values $\beta$ and $k$ on the condition number of the linearized system matrices. The condition numbers, measured with respect to the $L^1(\Omega)$ norm, are estimated using the MATLAB function \texttt{condest}, which is particularly suited for sparse matrices. Figure \ref{fig:cond:numbers} shows the computed condition numbers, averaged over all iterations, for several mesh sizes and different choices for the Forchheimer coefficient $\beta$ and the permeability parameter $k$. These results indicate a clear growth in the condition numbers as $k$ increases. In particular, for the largest tested value $k=10^{2}$, the $L$-scheme consistently yields the smallest condition numbers across all considered values of $\beta$, indicating a comparatively more favorable behaviour under strong nonlinear effects.

 \begin{figure}[t!]
 \begin{center}
 \hspace*{-1.7cm}
 \includegraphics[scale=0.33]{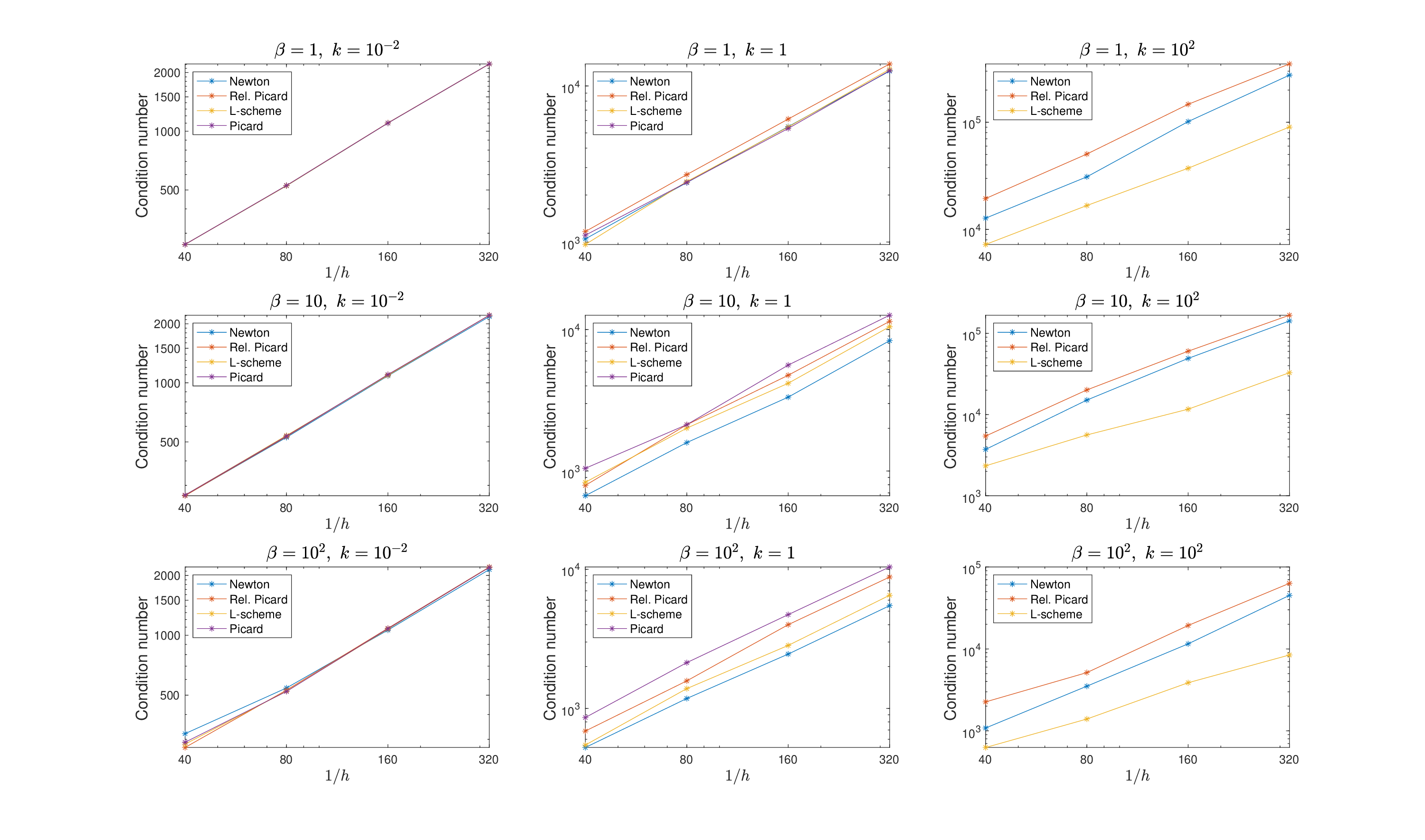}
 \end{center}
 \caption{Average condition numbers for several mesh sizes and different choices for the Forchheimer coefficient $\beta$ and the permeability constant $k$.}\label{fig:cond:numbers}
\end{figure}

\subsection{A test with discontinuous permeability}

Next, we consider problem \eqref{ibvp} posed on $\Omega\times (0,T)=(0,1)^2\times(0,1)$ with $f=0$ and Dirichlet boundary conditions for the pressure given by $p(x,y,t)=1-x$ on the entire $\partial\Omega$. The aim of this numerical example is to extend the comparison of the linearization methods to the solution of problems with abrupt variations in permeability. In particular, we consider several patterns of low-permeability regions: a vertical low-permeability strip, shown in gray in Figure \ref{fig:discont:perm:config} (left), and square- and L-shaped low-permeability inclusions, depicted in gray in Figure \ref{fig:discont:perm:config} (center) and (left), respectively. These types of discontinuous permeability configurations were previously considered in \cite{li:rui:2023,kum:rod:gas:oos:2019}. In all cases, the permeability coefficient is set to $k=10^{-4}$ inside the strip or inclusion and $k = 1$ in the rest of the domain. Moreover, the remaining problem parameters are set $\mu=\rho_{ref}=1$, $p_{ref}=0$ and $c_f=10^{-5}.$

For the iterative solvers, we again consider the $L$-scheme with $\gamma=0$ and a parameter $L$ that is proportional to $\beta^{1/2}$. The Picard relaxation parameter in this experiment is considered to be $\omega=0.8$.

\begin{figure}[t!]
\begin{center}
    \includegraphics[scale=0.17]{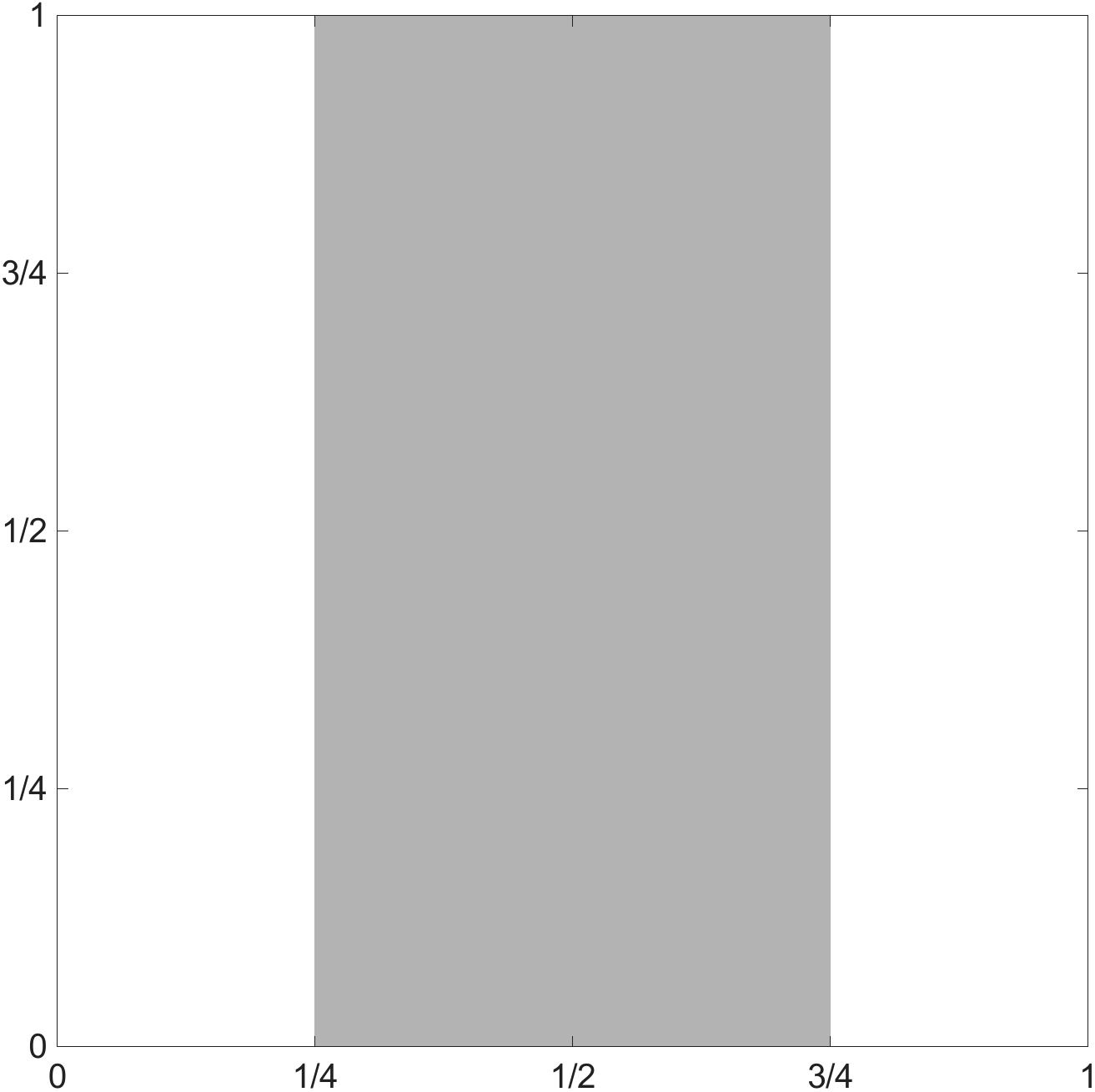}\qquad
    \includegraphics[scale=0.17]{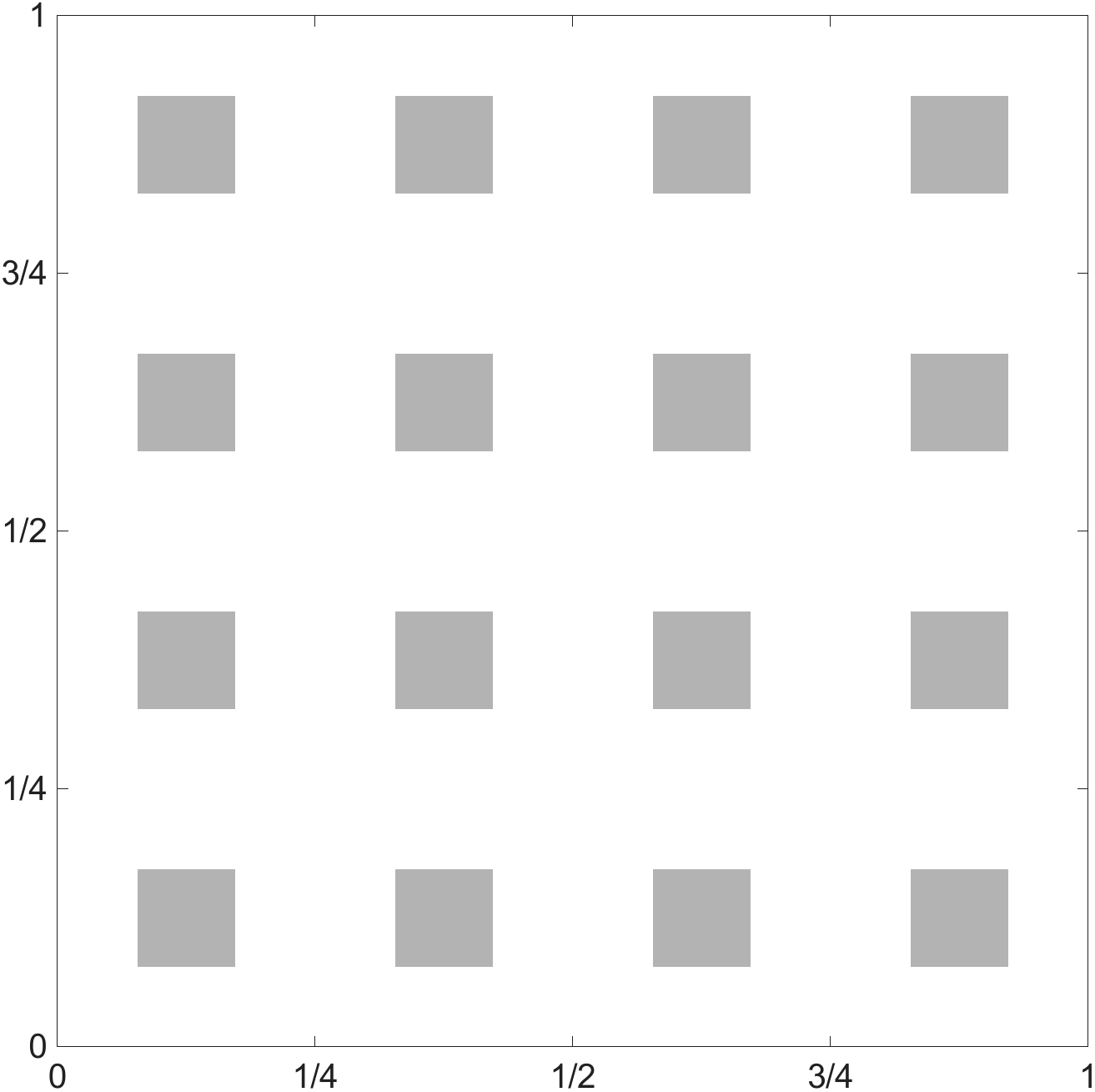}\qquad
    \includegraphics[scale=0.17]{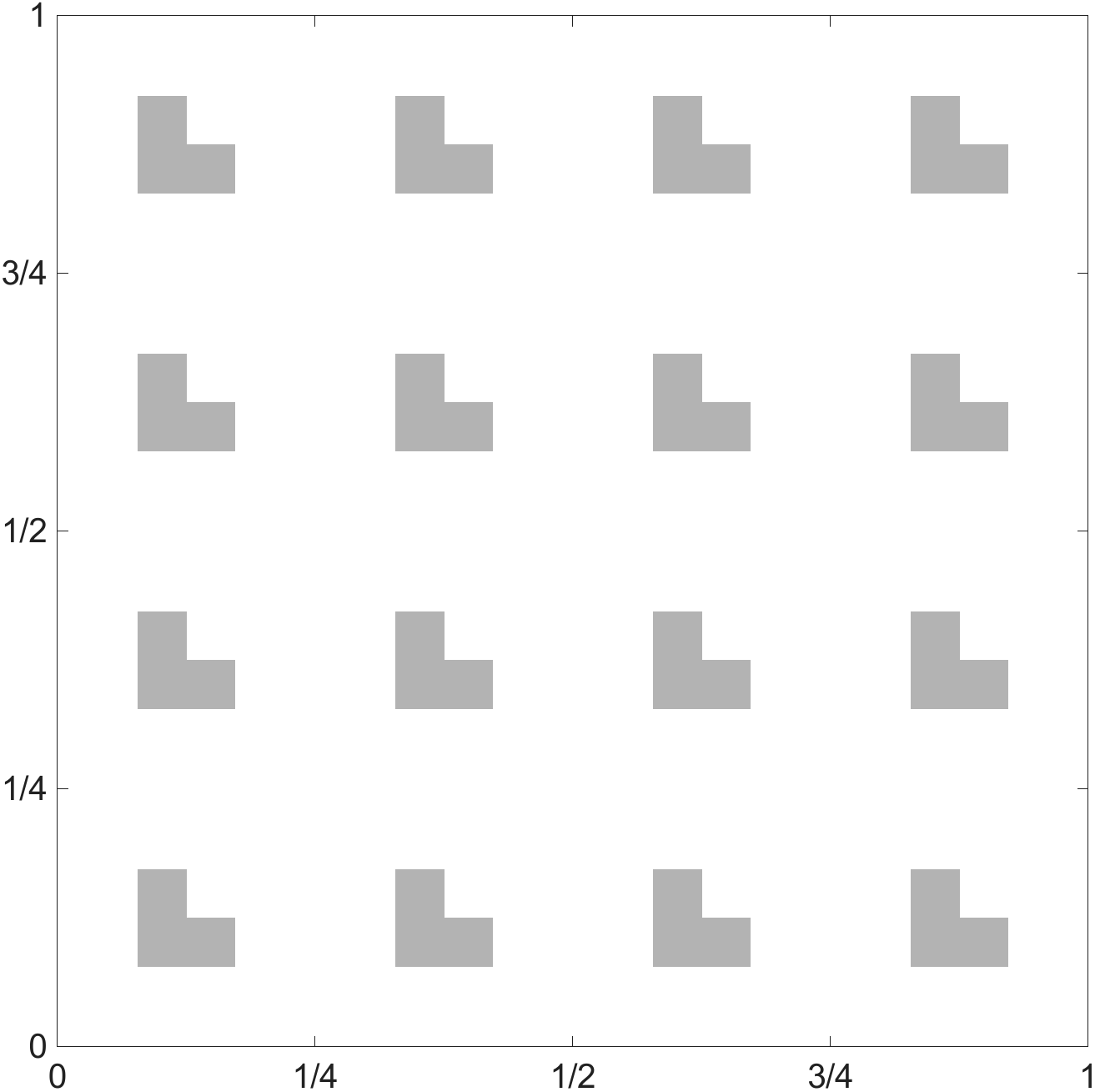}
    \end{center}
    \label{fig:discont:perm:config}\caption{Patterns of discontinuous permeability coefficients: strip (left), square inclusions (center) and L-shape inclusions (right). In the grey areas $k=10^{-4}$, everywhere else $k=1$.}
\end{figure}

Table \ref{table:discont:k} shows the average number of iterations and the total CPU time required by each method to converge, considering the stopping criterion \eqref{stopping:crit} with $\mathtt{tol}_a=\mathtt{tol}_r=10^{-5}$. Similar to the previous example, as $\beta$ grows, each method requires more iterations to converge. This is especially notorious in the first time step for all the methods. Once again, Picard method performs satisfactorily for $\beta=1$, but fails to converge for $\beta=10^4$. Newton method requires the smallest number of iterations to converge but, due to its significantly higher computational cost per iteration, it shows the highest CPU times. In this regard, the $L$-scheme remains the fastest method.

\begin{center}
		\begin{table}[t!]
			\begin{center}
				\begin{tabular}{|c|c|c|c|c|c|c|c|c|c|c|}\cline{3-10}
					\multicolumn{2}{c|}{}& \multicolumn{2}{c|}{Newton}  &  \multicolumn{2}{c|}{Picard}  &  \multicolumn{2}{c|}{Rel. Picard} &  \multicolumn{2}{c|}{$L$-scheme}\\\hline
					Pattern&$\beta$& It &  Time & It &  Time & It  & Time &It  & Time \\\hline
					\multirow{2}{*}{Strip} 
					&1  & 2.5 & 1.4e+1 & 4.3 & 5.0e+0 & 4.8 & 5.8e+0 &  3.9 &  4.6e+0 \\\cline{2-10}
					&$10^4$  & 3.2 & 1.6e+1 & $-$ & $-$ & 5.8 & 6.7e+0 & 3.8  & 4.5e+0  \\\hline
					\multirow{2}{*}{Squares} 
					&1  & 2.3 & 1.2e+1 & 3.3 & 3.9e+0 & 3.7 & 4.3e+0 &  3 & 3.5e+0  \\\cline{2-10}
					&$10^4$  & 3 & 1.5e+1 & $-$ & $-$& 5 & 5.9e+0 & 3.7  & 4.4e+0  \\\hline
					\multirow{2}{*}{L-shapes} 
					&1  & 2.3 & 1.2e+1 & 3.6 & 4.2e+0 & 3.8 & 4.5e+0 & 3.2  &  3.8e+0 \\\cline{2-10}
					&$10^4$ & 3 & 1.6e+1 & $-$ & $-$ & 5.2 & 6.1e+0 & 3.8  & 4.5e+0  \\\hline
				\end{tabular}
			\end{center}
			\caption{Time step $\tau=0.1$, mesh size $h=1/160$, $c_f=10^{-5}$, Picard relaxation parameter $\omega=0.8$, $L$-scheme ($\gamma=0$) parameter $L=0.7$ for $\beta=1$, and $L=70$ for $\beta=10^4$.}\label{table:discont:k}
		\end{table}
	\end{center}

    \begin{figure}[t]
    \begin{center}\vspace*{-1cm}\hspace*{-0.4cm}
    \includegraphics[scale=0.19]{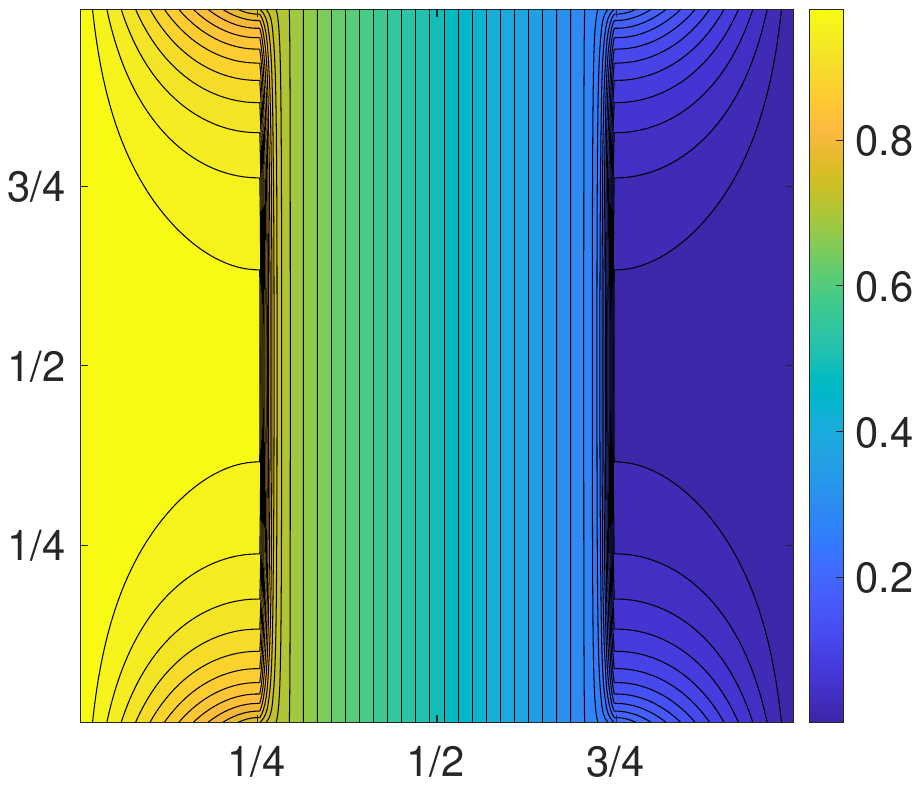}\hspace*{-0.8cm}
    \includegraphics[scale=0.19]{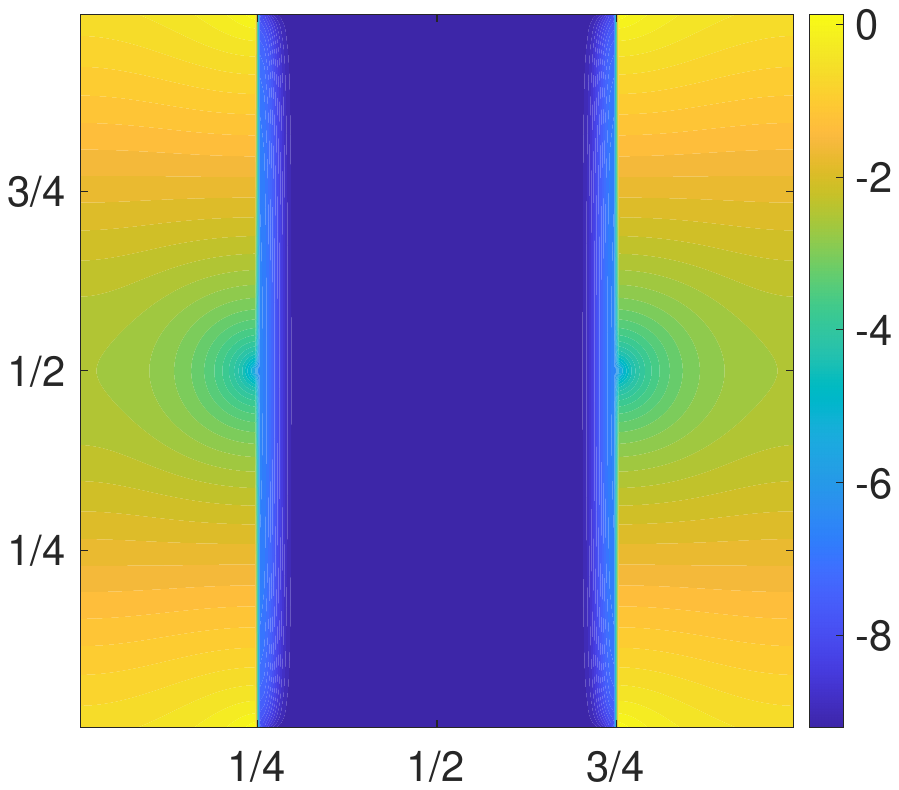}\hspace*{-0.8cm}
    \includegraphics[scale=0.19]{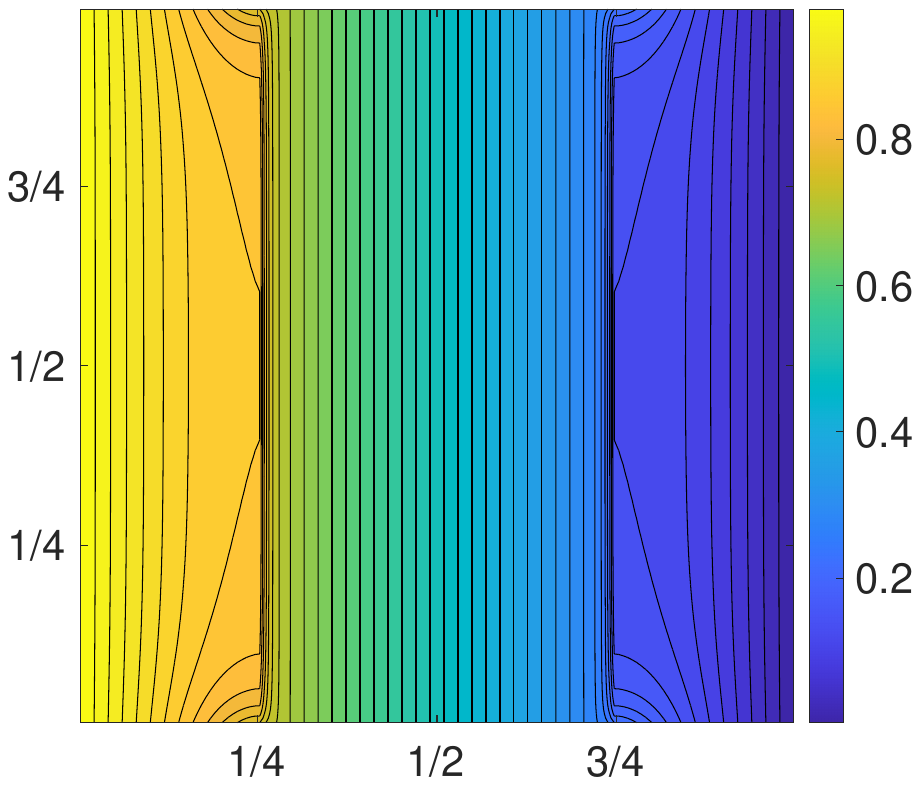}\hspace*{-0.8cm}
    \includegraphics[scale=0.19]{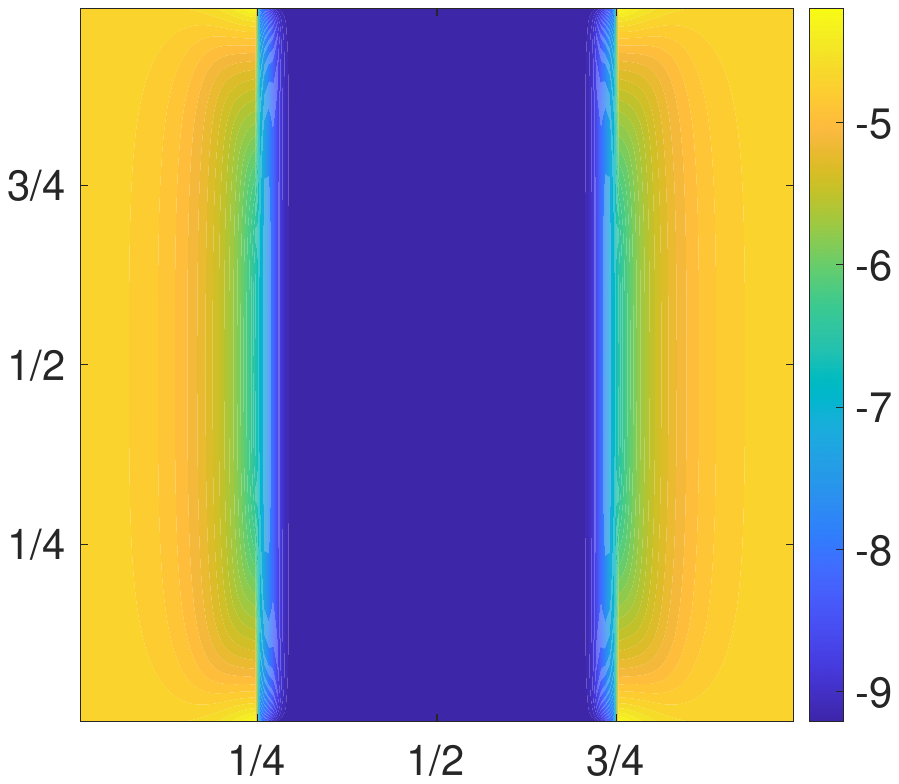}\\[-1.5cm]
    \hspace*{-0.4cm}\footnotesize(a)\hspace*{2.8cm} (b) \hspace*{2.65cm} (c) \hspace*{2.65cm} (d)
    \end{center}
    \caption{(a) Numerical pressure and (b) logarithm of the norm of the numerical velocity for the strip configuration and $\beta=1$. Accordingly, graphs (c) and (d) correspond to the case $\beta=10^4$.}\label{fig:pres:vel:strip}\end{figure}

    \begin{figure}[h!]
    \begin{center}\vspace*{-1cm}\hspace*{-0.4cm}
    \includegraphics[scale=0.19]{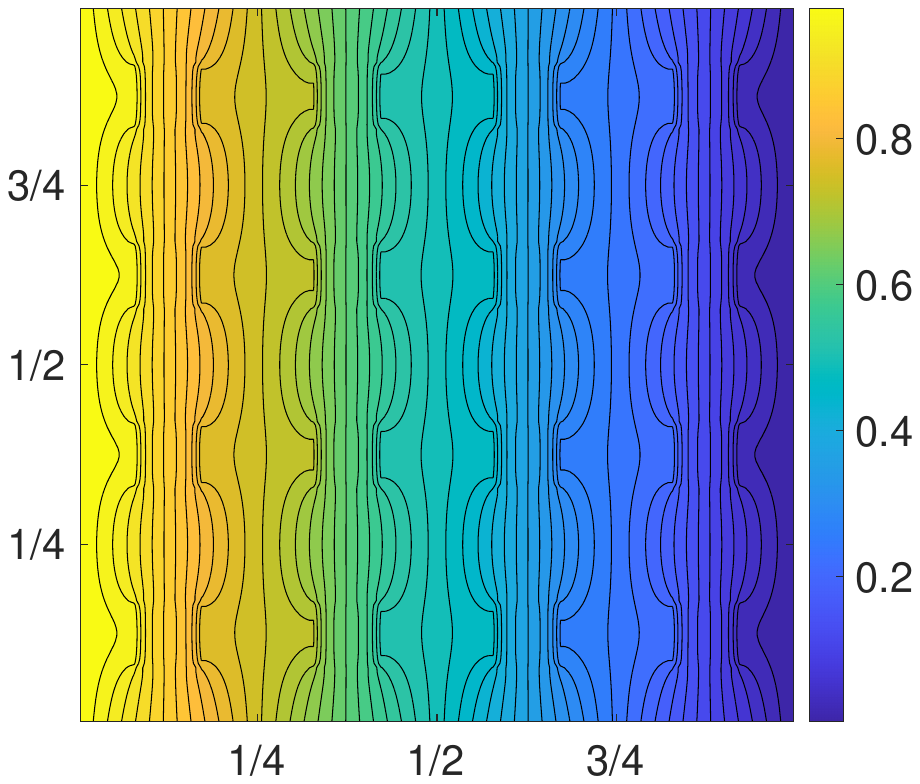}\hspace*{-0.8cm}
    \includegraphics[scale=0.19]{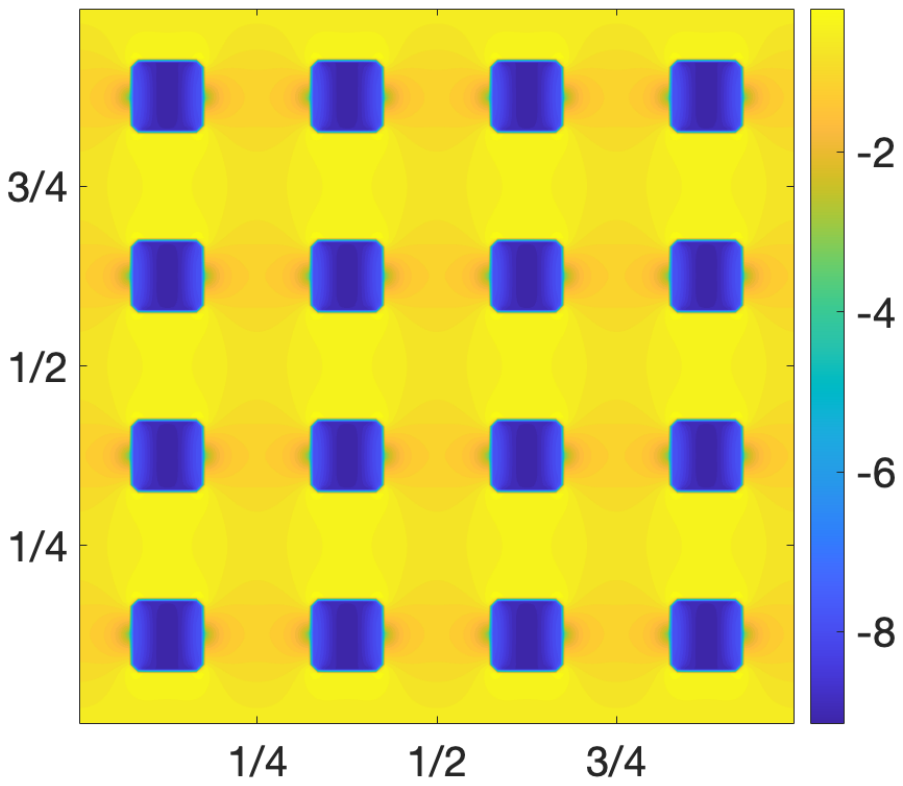}\hspace*{-0.8cm}
    \includegraphics[scale=0.19]{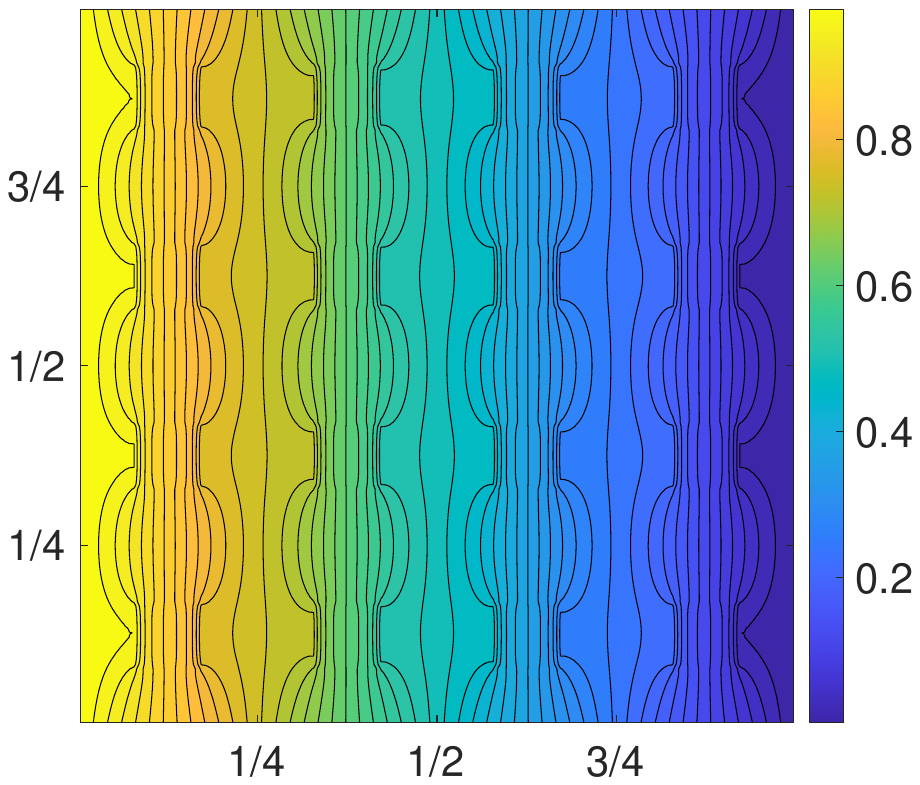}\hspace*{-0.8cm}
    \includegraphics[scale=0.19]{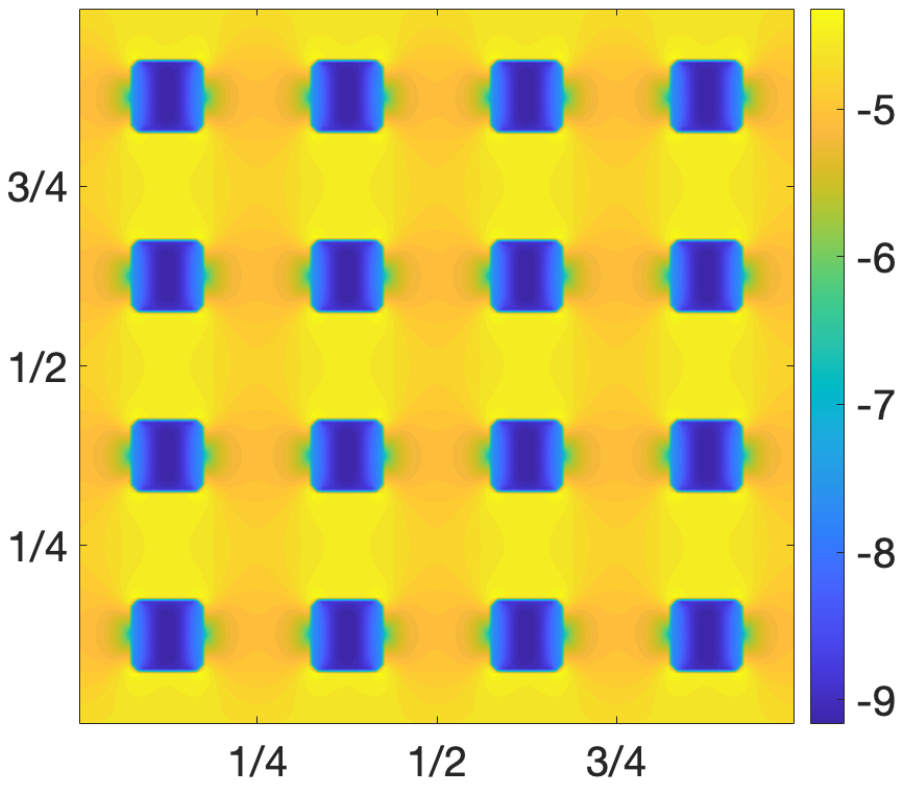}\\[-1.5cm]
    \hspace*{-0.4cm}\footnotesize(a)\hspace*{2.8cm} (b) \hspace*{2.65cm} (c) \hspace*{2.65cm} (d)
    \end{center}
    \caption{(a) Numerical pressure and (b) logarithm of the norm of the numerical velocity for the square inclusion configuration and $\beta=1$. Accordingly, graphs (c) and (d) correspond to the case $\beta=10^4$.}\label{fig:pres:vel:squares}\end{figure}

    \begin{figure}[h!]
    \begin{center}\vspace*{-1cm}\hspace*{-0.5cm}
    \includegraphics[scale=0.19]{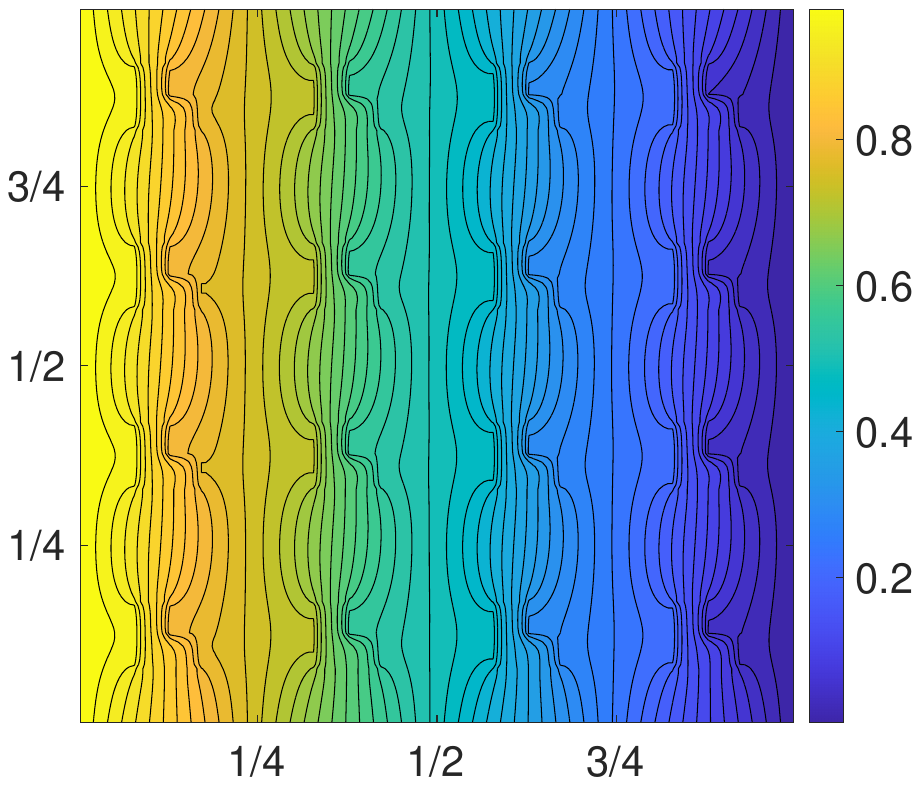}\hspace*{-0.8cm}
    \includegraphics[scale=0.19]{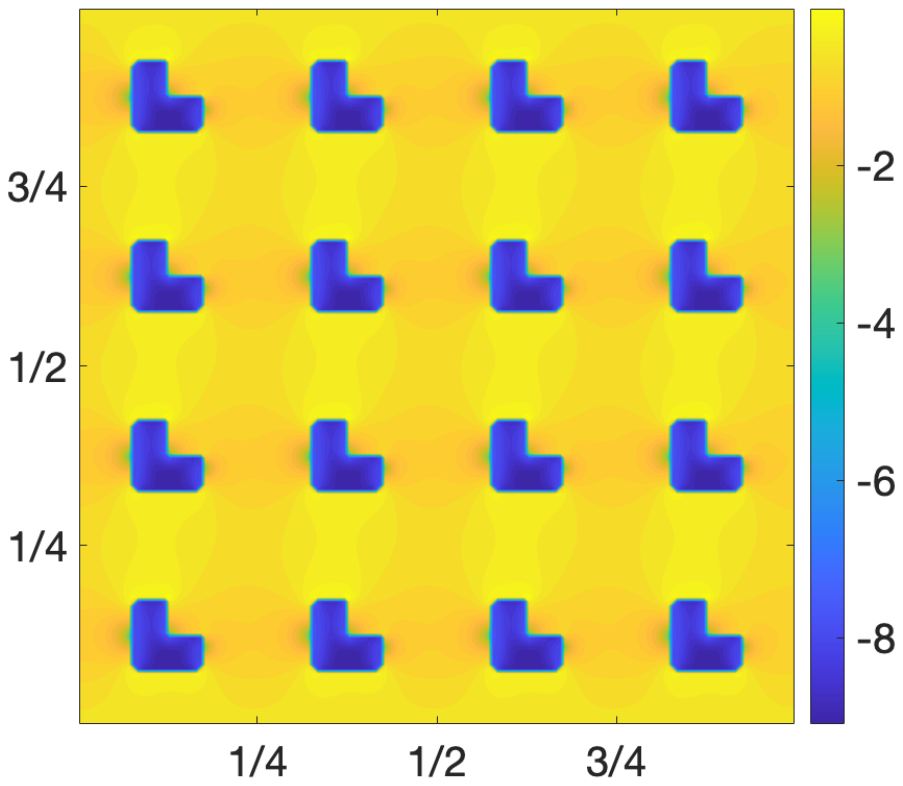}\hspace*{-0.8cm}
    \includegraphics[scale=0.19]{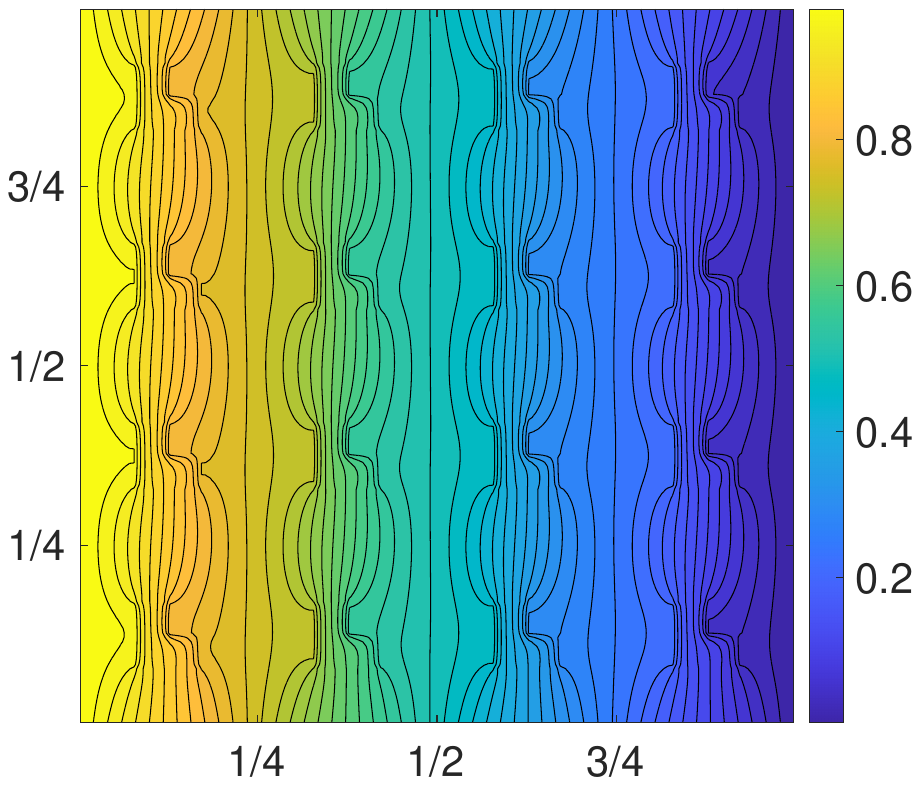}\hspace*{-0.8cm}
    \includegraphics[scale=0.19]{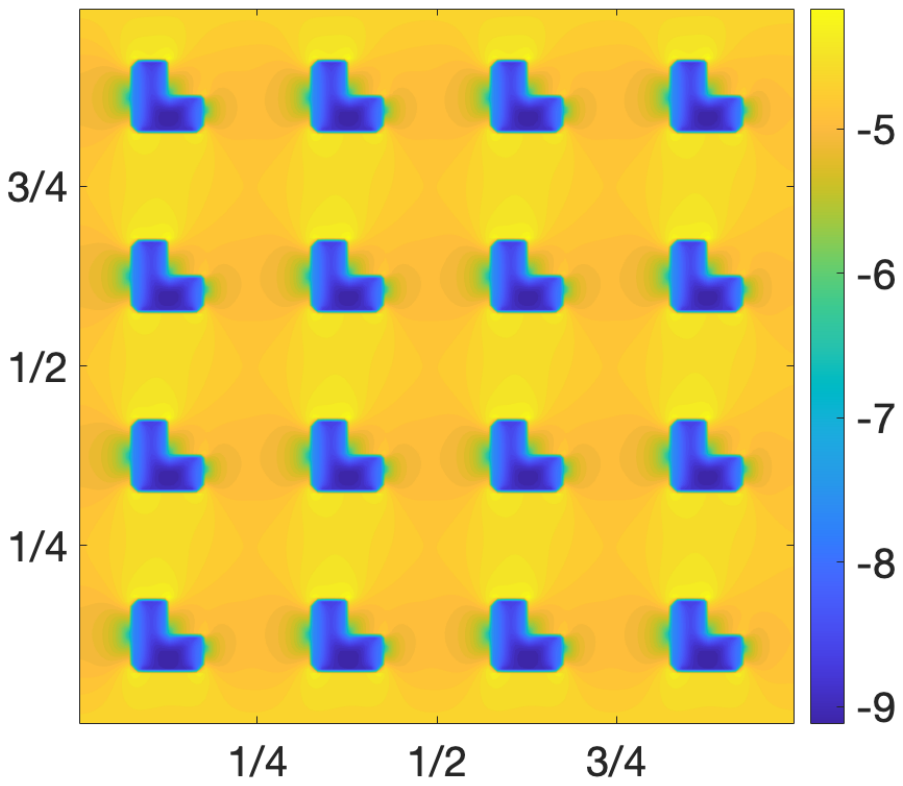}\\[-1.5cm]
    \hspace*{-0.4cm}\footnotesize(a)\hspace*{2.8cm} (b) \hspace*{2.65cm} (c) \hspace*{2.65cm} (d) 
    \end{center}
    \caption{(a) Numerical pressure and (b) logarithm of the norm of the numerical velocity for the L-shape inclusion configuration and $\beta=1$. Accordingly, graphs (c) and (d) correspond to the case $\beta=10^4$.}\label{fig:pres:vel:L:shape}\end{figure}

Figures \ref{fig:pres:vel:strip}-\ref{fig:pres:vel:L:shape} show the numerical pressures and the logarithm of the norm of the numerical velocities at $t=1$, computed using the $L$-scheme for the different discontinuous permeability configurations. In all cases, the computations are performed with a time step $\tau=0.1$ and a mesh size $h = 1/160$. In each figure, plots (a) and (b) correspond to $\beta=1$, while (c) and (d) correspond to $\beta=10^4$.  The numerical solutions in all configurations show the expected physical behaviour, capturing the effects of the low-permeability regions on both the pressure and velocity fields.
    
\section{Conclusions} \label{sec:conclusion}
In this paper, we considered a slightly compressible Darcy-Forchheimer model. The model has many practical applications connected to combustion in porous media and is highly nonlinear. A backward Euler time integration was combined with a mixed finite element spatial discretization to provide numerical approximations to the exact solution. In this framework, we proposed a family of robust linearization schemes, for which a (free to be chosen) tuning parameter is involved, in accordance with the $L$-scheme. We showed the theoretical convergence of such schemes under certain assumptions, and compared them with existing linearization techniques, such as Picard, relaxed Picard and Newton methods. In this respect, a series of numerical experiments were performed to illustrate the convergence and robustness of the proposed algorithms.

\section*{Acknowledgments}
The work of A. Arrarás, L. Portero and F.J. Gaspar was supported by Grant PID2022-140108NB-I00 funded by MCIN/AEI/10.13039/501100011033 and by ``ERDF A way of making Europe''. F.A. Radu acknowledges the
financial support from the Vista Center for Modeling of Coupled Subsurface Dynamics
(VISTA CSD) and the support from the project MUPSI, CETP-2023-00298.

\bibliographystyle{elsarticle-num} 
  \bibliography{bibliography}

\begin{thebibliography}{10}
\expandafter\ifx\csname url\endcsname\relax
  \def\url#1{\texttt{#1}}\fi
\expandafter\ifx\csname urlprefix\endcsname\relax\def\urlprefix{URL }\fi
\expandafter\ifx\csname href\endcsname\relax
  \def\href#1#2{#2} \def\path#1{#1}\fi

\bibitem{gir:whe:2008}
V.~Girault, M.~F. Wheeler, Numerical discretization of a {D}arcy-{F}orchheimer
  model, Numer. Math. 110~(2) (2008) 161--198.

\bibitem{lop:mol:sal:2009}
H.~L{\'o}pez, B.~Molina, J.~J. Salas, Comparison between different numerical
  discretizations for a {D}arcy-{F}orchheimer model, Electron. Trans. Numer.
  Anal 34 (2009) 187--203.

\bibitem{aul:blo:hoa:ibra:2009}
E.~Aulisa, L.~Bloshanskaya, L.~Hoang, A.~Ibragimov, Analysis of generalized
  {F}orchheimer flows of compressible fluids in porous media, J. Math. Phys.
  50~(10) (2009).

\bibitem{kie:2015}
T.~T. Kieu, Analysis of expanded mixed finite element methods for the
  generalized {F}orchheimer flows of slightly compressible fluids, Numer.
  Methods Partial Differ. Equ. 32~(1) (2015) 60–85.

\bibitem{rui:pan:2017}
H.~Rui, H.~Pan, A block-centered finite difference method for slightly
  compressible {D}arcy-{F}orchheimer flow in porous media, J. Sci. Comput.
  73~(1) (2017) 70–92.

\bibitem{CombustionReview}
M.~A. Mujeebu, M.~Z. Abdullah, M.~Z.~A. Bakar, A.~A. Mohamad, R.~M.~N. Muhad,
  M.~K. Abdullah, Combustion in porous media and its applications – {A}
  comprehensive survey, J. Environ. Manag. 90~(8) (2009) 2287--2312.

\bibitem{Fabrie_1989}
P.~Fabrie, Regularity of the solution of {D}arcy-{F}orchheimer’s equation,
  Nonlinear Anal. 13~(9) (1989) 1025–1049.

\bibitem{Amirat_1991}
Y.~Amirat, Écoulements en milieu poreux n’obéissant pas à la loi de
  {D}arcy, ESAIM Math. Model. Numer. Anal. 25~(3) (1991) 273–306.

\bibitem{kna:sum:2016}
P.~Knabner, G.~Summ, Solvability of the mixed formulation for
  {D}arcy-{F}orchheimer flow in porous media (2016).
\newblock \href {http://arxiv.org/abs/1608.08829} {\path{arXiv:1608.08829}}.

\bibitem{douglas1993}
J.~Douglas, Jr., P.~J. Paes-Leme, T.~Giorgi, Generalized {F}orchheimer flow in
  porous media, in: Lions, J.-L., Baiocchi, C. (eds.), Boundary Value Problems
  for Partial Differential Equations and Applications, Vol.~29 of RMA Res.
  Notes Appl. Math., Masson, Paris, 1993, pp. 99--111.

\bibitem{kim1999}
M.~Y. Kim, E.~J. Park, Fully discrete mixed finite element approximations for
  non-{D}arcy flows in porous media, Comput. Math. Appl. 38~(11-12) (1999)
  113--129.

\bibitem{ewing1999}
R.~E. Ewing, R.~D. Lazarov, S.~L. Lyons, D.~V. Papavassilou, J.~Pasciak,
  G.~Qin, Numerical well model for non-{D}arcy flow through isotropic porous
  media, Comput. Geosci. 3~(3-4) (1999) 184--204.

\bibitem{summ_thesis}
G.~Summ, Lösbarkeit un {D}iskretisierung der gemischten {F}ormulierung für
  {D}arcy-{F}orchheimer-{F}luss in porösen {M}edien, Ph.D. thesis,
  Friedrich-Alexander-Universität Erlangen-Nürnberg (2001).

\bibitem{park2005}
E.~J. Park, Mixed finite element methods for generalized {F}orchheimer flow in
  porous media, Numer. Methods Partial Differ. Equ. 21~(2) (2005) 213--228.

\bibitem{peaceman1955}
D.~H. Peaceman, H.~H. Rachford, The numerical solution of parabolic and
  elliptic differential equations, J. Soc. Ind. Appl. Math. 3~(1) (1955)
  28--41.

\bibitem{pan:rui:2012}
H.~Pan, H.~Rui, Mixed element method for two-dimensional {D}arcy-{F}orchheimer
  model, J. Sci. Comput. 52~(3) (2012) 563--587.

\bibitem{rui:pan:2012}
H.~Rui, H.~Pan, A block-centered finite difference method for the
  {D}arcy-{F}orchheimer model, SIAM J. Numer. Anal. 50~(5) (2012) 2612–2631.

\bibitem{wei:whe:1988}
A.~Weiser, M.~F. Wheeler, On convergence of block-centered finite differences
  for elliptic problems, SIAM J. Numer. Anal. 25~(2) (1988) 351--375.

\bibitem{arb:whe:yot:1997}
T.~Arbogast, M.~F. Wheeler, I.~Yotov, Mixed finite elements for elliptic
  problems with tensor coefficients as cell-centered finite differences, SIAM
  J. Numer. Anal. 34~(2) (1997) 828--852.

\bibitem{xu:lia:rui:2017}
W.~Xu, D.~Liang, H.~Rui, A multipoint flux mixed finite element method for the
  compressible {D}arcy-{F}orchheimer models, Appl. Math. Comput. 315 (2017)
  259–277.

\bibitem{both2020DF}
J.~W. Both, J.~M. Nordbotten, F.~A. Radu, Free energy diminishing
  discretization of {D}arcy-{F}orchheimer flow in poroelastic media, in:
  Klöfkorn, R., Keilegavlen, E., Radu, F. A., Fuhrmann, J. (eds.), Finite
  Volumes for Complex Applications IX - Methods, Theoretical Aspects, Examples,
  Vol. 323 of Springer Proc. Math. Stat., Springer, Cham, 2020, pp. 203--211.

\bibitem{bothGradientFlow2019}
J.~W. Both, K.~Kumar, J.~M. Nordbotten, F.~A. Radu, The gradient flow
  structures of thermo-poro-visco-elastic processes in porous media (2019).
\newblock \href {http://arxiv.org/abs/1907.03134} {\path{arXiv:1907.03134}}.

\bibitem{kna:rob:2014}
P.~Knabner, J.~E. Roberts, Mathematical analysis of a discrete fracture model
  coupling {D}arcy flow in the matrix with {D}arcy-{F}orchheimer flow in the
  fracture, ESAIM Math. Model. Numer. Anal. 48~(5) (2014) 1451--1472.

\bibitem{Frih_2006}
N.~Frih, J.~E. Roberts, A.~Saada, Un modèle {D}arcy-{F}orchheimer pour un
  écoulement dans un milieu poreux fracturé, ARIMA 5 (2006) 129–143.

\bibitem{Frih_2008}
N.~Frih, J.~E. Roberts, A.~Saada, Modeling fractures as interfaces: a model for
  {F}orchheimer fractures, Comput. Geosci. 12~(1) (2008) 91–104.

\bibitem{arr:gas:por:rod:2019}
A.~Arrarás, F.~Gaspar, L.~Portero, C.~Rodrigo, Geometric multigrid methods for
  {D}arcy-{F}orchheimer flow in fractured porous media, Comput. Math. Appl.
  78~(9) (2019) 3139--3151.

\bibitem{huang:chen:rui:2018}
J.~Huang, L.~Chen, H.~Rui, Multigrid methods for a mixed finite element method
  of the {D}arcy-{F}orchheimer model, J. Sci. Comput. 74 (2018) 396--411.

\bibitem{rui:liu:2015}
H.~Rui, W.~Liu, A two-grid block-centered finite difference method for
  {D}arcy-{F}orchheimer flow in porous media, SIAM J. Numer. Anal. 53~(4)
  (2015) 1941--1962.

\bibitem{listradu2016}
F.~List, F.~A. Radu, A study on iterative methods for solving {R}ichards’
  equation, Comput. Geosci. 20 (2016) 341--353.

\bibitem{pop2004}
I.~S. Pop, F.~A. Radu, P.~Knabner, Mixed finite elements for the {R}ichards’
  equation: linearization procedure, J. Comput. Appl. Math. 168~(1-2) (2004)
  365--373.

\bibitem{al2024iterative}
R.~Al~Dbaissy, T.~Sayah, Iterative scheme for the {D}arcy-{F}orchheimer problem
  with pressure boundary condition, WSEAS Transactions on Heat and Mass
  Transfer 19 (2024) 94--106.

\bibitem{sayah2021}
T.~Sayah, G.~Semaan, F.~Triki, Finite element methods for the
  {D}arcy-{F}orchheimer problem coupled with the convection-diffusion-reaction
  problem, ESAIM Math. Model. Numer. Anal. 55 (2021) 2643--2678.

\bibitem{dou:rob:1983}
J.~Douglas, J.~E. Roberts, Numerical methods for a model for compressible
  miscible displacement in porous media, Math. Comput. 41~(164) (1983)
  441--459.

\bibitem{rut:ma:1992}
D.~Ruth, H.~Ma, On the derivation of the {F}orchheimer equation by means of the
  averaging theorem, Transp. Porous Med. 7~(3) (1992) 255--264.

\bibitem{kieu:2020}
T.~Kieu, Existence of a solution for generalized {F}orchheimer flow in porous
  media with minimal regularity conditions, J. Math. Phys. 61~(1) (2020).

\bibitem{li:rui:2023}
H.~Li, H.~Rui, Parameter-robust mixed element method for poroelasticity with
  {D}arcy-{F}orchheimer flow, Numer. Methods Partial Differ. Equ. 39~(5) (2023)
  3634--3656.

\bibitem{kum:rod:gas:oos:2019}
P.~Kumar, C.~Rodrigo, F.~J. Gaspar, C.~W. Oosterlee, On local {F}ourier
  analysis of multigrid methods for {PDE}s with jumping and random
  coefficients, SIAM J. Sci. Comput. 41~(3) (2019) A1385--A1413.

\end{thebibliography}


\appendix\section{}\label{sec:appendix}

In this appendix, we consider a simplified Darcy-Forchheimer model in which the density $ \rho $ is assumed to be constant. For simplicity, we assume that $ \rho = 1$. We state a general linearization scheme depending on a parameter $ \gamma \in [0,1]$ and prove its convergence. The case $ \gamma = 0$ was already treated in \cite{al2024iterative}; however, the proof presented here is slightly different.

 {\bf Linearization scheme for the simplified Darcy-Forchheimer model.} Let $L \ge 0$ be arbitrary, $ \gamma \in [0,1]$, $ {\bf u}^{n,0} := {\bf u}^{n-1}$, $p^{n,0} := p^{n-1}$ and $i \ge 1$. Given ${\bf u}^{n, i-1}$, find $(\ui, p^{n,i})$ such that, for all $\bv \in V$ and $q \in W$, it holds
\begin{subequations}
 \begin{align}
     \la \mu k^{-1} \ui, \bv \ra + \beta\,\la |{\bf u}^{n,i-1}|\,(\gamma {\bf u}^{n,i-1} + (1-\gamma) \ui), \bv \ra \nonumber\\
    \hspace*{0cm}+\,L\, \la \ui - {\bf u}^{n,i-1}, \bv \ra - \la p^{n,i}, \nabla \cdot \bv \ra &= \la \mathbf{0},\mathbf{v} \ra,\label{eq:appendix:1}\\[1ex]
\la p^{n,i}, q \ra + \tau \la \nabla \cdot \ui, q \ra &= \la p^{n,0}, q \ra+\tau\,\la f^n, q \ra.
\label{eq:appendix:2}
 \end{align}   
\end{subequations}
The main result regarding the convergence of the proposed linearization schemes now follows.
\begin{theorem} \label{theorem_conv_appendix} Assuming that (A1)-(A2) hold true, the iterative scheme \eqref{eq:appendix:1}-\eqref{eq:appendix:2} converges if the parameter {\it L} satisfies
\begin{equation} \label{convergence_condition_appendix}
    L > \dfrac{k\beta^2 M_u^2\, (1 + \gamma)^2 }{2\mu}.
\end{equation}
\end{theorem}
{\bf Proof.} We use the same notations as in Section \ref{sec:linschemes} for the errors at the $i$th iteration, i.e.,
\begin{align*}
 \eiu & := {\bu}^{n,i} - {\bu}^n, \\
 \epi & := p^{n,i} - p^n. 
\end{align*}
By subtracting the equations \eqref{eq:discrete_u} and \eqref{eq:discrete_p} (with a constant density $ \rho = 1$) from \eqref{eq:appendix:1} and \eqref{eq:appendix:2}, respectively, we get
\begin{align*}
  \langle \mu k^{-1} \eiu, \bv \rangle  + \beta\,\la |{\bf u}^{n, i-1}|\,(\gamma {\bf u}^{n,i-1} + (1-\gamma)\, \ui) - |{\bf u}^{n}|\,{\bf u}^{n}, \bv \rangle \nonumber\\
  +\, L \,\la \eiu - \fui, \bv \rangle -  \langle \epi, \nabla \cdot \bv \ra & = 0
\end{align*}
and
\begin{equation*}
  \la \epi, q \rangle + \tau\,\la \nabla \cdot \eiu, q \ra = 0 , 
\end{equation*}
  for all $\bv \in V$ and $q \in W$. Now, we test the above expressions with $\bv = \tau \eiu$ and $ q = \epi $, respectively, and add the resulting equations to obtain
\begin{align*}
    \tau\mu\,k^{-1} \| \eiu \|^{2} + \tau \beta\, \la |{\bf u}^{n, i-1}|\, \gamma\, ({\bf u}^{n, i-1} - {\bf u}^{n, i}) , \eiu \rangle \nonumber \\
    +\, \tau \beta\, \la |{\bf u}^{n, i-1}|\, {\bf u}^{n,i} - |{\bf u}^{n}|\, {\bf u}^{n}, \eiu \rangle 
    + \tau L\,\langle \eiu - \fui, \eiu \rangle + \| \epi \|^{2} &= 0.
\end{align*}
This is further equivalent to 
\begin{align}
    \tau\mu k^{-1} \| \eiu \|^{2} + \tau \beta \gamma\, \la |{\bf u}^{n, i-1}|\,(\fui - \eiu ), \eiu \rangle + \tau \beta\, \la |{\bf u}^{n, i-1}|\, {\bf u}^{n,i} - |{\bf u}^{n}|\, {\bf u}^{n}, \eiu \rangle \nonumber \\
    + \dfrac{\tau L }{2}\,\| \eiu \|^{2}  + \dfrac{\tau L }{2}\, \| \eiu - \fui \|^{2} +  \| \epi \|^{2} = \dfrac{\tau L }{2}\,\| \fui \|^{2}.\label{eq:proof:appendix:2}
\end{align}
Moreover, by some algebraic manipulations, the previous expression becomes
\begin{align}
    \tau\mu k^{-1} \| \eiu \|^{2} + \tau \beta \gamma\, \la |{\bf u}^{n, i-1}|\, (\fui - \eiu), \eiu \rangle \nonumber \\ 
    +\, \tau \beta\, \la (|{\bf u}^{n, i-1}|  - |{\bf u}^{n, i}|)\, {\bf u}^{n,i}, \eiu \rangle + \tau \beta\, \la |{\bf u}^{n, i}|\,{\bf u}^{n,i}  - |{\bf u}^{n}|\, {\bf u}^{n}, \eiu \rangle \nonumber \\
    +\, \dfrac{\tau L }{2}\,\| \eiu \|^{2}  + \dfrac{\tau L }{2}\, \| \eiu - \fui \|^{2} +  \| \epi \|^{2} &= \dfrac{\tau L }{2}\,\| \fui \|^{2}.\label{eq:proof:appendix:3}
\end{align}
Now, by using the inequality (see, e.g., \cite{kna:sum:2016}; it follows from \eqref{lemma:ineq:2})
\begin{equation*} \label{conv:proof:2}
\la |{\bf u}^{n, i}|\,{\bf u}^{n,i}  - |{\bf u}^{n}|\, {\bf u}^{n}, \eiu \rangle \ge \dfrac{1}{2} \,\| \eiu \|^{3}_{L^{3}(\Omega)},
\end{equation*}
 and denoting 
\begin{displaymath}
 T =  -\tau \beta \gamma\, \la |{\bf u}^{n, i-1}|\, (\eiu - \fui), \eiu \rangle -\tau \beta\, \la (|{\bf u}^{n, i-1}|  - |{\bf u}^{n, i}|)\, {\bf u}^{n,i}, \eiu \rangle,  
\end{displaymath}
we obtain, from \eqref{eq:proof:appendix:3},
\begin{equation} \label{conv:proof:appendix:3}
\begin{array}{c}
    \left(\dfrac{\tau L}{2} +  \tau\,\mu\, k^{-1}\right)  \|  \eiu \|^{2}  + \dfrac{\tau \beta}{2}\,\| \eiu \|^{3}_{L^{3}(\Omega)} + \dfrac{\tau L}{2}\, \| \eiu - \fui \|^{2} +\| \epi \|^2\\
    \hfill\le \dfrac{\tau L }{2}\,\| \fui \|^{2} + T.
    \end{array}
\end{equation}
We now obtain an estimate for the term $T$ above. By using the assumption (A2) and the Cauchy-Schwarz inequality, we get
\begin{align}
 | T |  &\leq \tau \beta M_u\, (1+\gamma)\, \| \eiu - \fui  \|\, \| \eiu \| \nonumber \\[1ex]
     &\leq \dfrac{\tau L}{2}\, \| \eiu - \fui  \|^{2}  + \dfrac{\tau \beta^2  M_u^2\, (1 + \gamma)^2}{2 L}\, \| \eiu \|^2. \label{conv:proof:appendix:4}
\end{align}
In the last step, we have used the Young inequality. Merging \eqref{conv:proof:appendix:3} and \eqref{conv:proof:appendix:4} together, we get
\begin{equation*} \label{conv:proof:appendix:5}
    \left(\dfrac{\tau L}{2} +  \tau\mu k^{-1} -  \dfrac{\tau \beta^2 M_u^2\, (1 + \gamma)^2}{2 L}\right)\|  \eiu \|^{2}  + \dfrac{\tau \beta}{2}\, \| \eiu \|^{3}_{L^{3}(\Omega)} +  \| \epi \|^2 \le \dfrac{\tau L }{2}\,\| \fui \|^{2}. 
\end{equation*} 
Hence, we have a contraction if
\begin{displaymath} 
  \tau\mu k^{-1} > \dfrac{\tau \beta^2 M_u^2\,(1 + \gamma)^2}{2 L},  \end{displaymath}
which is nothing else than \eqref{convergence_condition_appendix}.
This completes the proof.\\
\mbox{ }\hfill {\bf Q.E.D.}

\end{document}